\documentclass[11pt,a4paper]{article}
\usepackage{amsmath,amsfonts,amssymb,amsthm}
\usepackage{epsfig}
\usepackage{times}
\newcommand{\I}{{\bf 1}}
\newtheorem{proposition}{Proposition}[section]
\newtheorem{theorem}[proposition]{Theorem}

\newtheorem{lemma}[proposition]{Lemma}
\newtheorem{remark}[proposition]{Remark}

\renewcommand\d{{\rm d}}
\renewcommand\o{{\rm o}}
\newcommand{\e}{{\rm e}}
\newcommand{\nc}{\newcommand}
\nc{\R}{{\mathbb R}}
\nc{\N}{{\mathbb N}}
\nc{\Z}{{\mathbb Z}}
\DeclareMathOperator{\card}{card}

\DeclareMathOperator{\inter}{int}
\nc{\BP}{\mathbb{P}}
\nc{\BQ}{\mathbb{Q}}
\nc{\BE}{\mathbb{E}}
\nc{\BS}{\mathbb{S}}
\nc{\G}{{\mathbb G}}
\renewcommand{\H}{{\mathcal H}}

\newcommand\var{\mathop{\rm Var}\nolimits}
\newcommand\cov{\mathop{\rm Cov}\nolimits}
\newcommand\dsum{\mathop{\sum\sum}\limits}
\newcommand\longnd{\mathop{\stackrel{{\rm d}}{\longrightarrow}}\limits_{n \to \infty}}
\newcommand\longm{\mathop{\longrightarrow}\limits_{m \to \infty}}
\newcommand\longn{\mathop{\longrightarrow}\limits_{n \to \infty}}

\begin{document}
\author{Daniel Hug, G{\"u}nter Last, Zbyn\v ek Pawlas and Wolfgang Weil}
\title{Statistics for   Poisson   models of overlapping spheres} 
\date{\today}
\maketitle
\begin{abstract} 
\noindent 
The paper considers the stationary Poisson Boolean model with spherical
grains and proposes a family of nonparametric estimators
for the radius distribution. These estimators
are based on observed distances and radii,
weighted in an appropriate way. They are ratio-unbiased
and asymptotically consistent for growing observation
window. It is shown that the
asymptotic variance exists and is given by a fairly explicit integral
expression. Asymptotic normality is established under a suitable
integrability assumption on the weight function.
The paper also provides a short discussion of
related estimators as well as a simulation study.
\end{abstract}
\renewcommand{\thefootnote}{{}}

%%-------------------------------------------------------
%%-------------------------------------------------------
\section{Introduction}
We consider a stationary random closed set $Z$ in $\R^d$ 
($d\ge 2$) which is given as a union of random balls of the form
\begin{equation}\label{eq:1.1}
Z:=\bigcup_{n\ge 1}B(\xi_n,R_n),
\end{equation}
where $B(x,r)$ is a closed Euclidean ball with radius $r\ge 0$ centered at $x\in\R^d$,
$\Phi:=\{\xi_n:n\ge 1\}$ is a stationary Poisson point process on
$\R^d$, and the sequence $(R_n)_{n\ge 1}$ is independent of $\Phi$
and is formed by independent non-negative random variables 
with common distribution $\G$. Let $R$ be a generic random variable
with distribution $\G$. We will always assume that it has a finite $2d$-th moment, that is, 
\begin{align}\label{R2d}
\BE R^{2d} <\infty.
\end{align}
Definition \eqref{eq:1.1} provides an important model
in stochastic geometry with numerous applications in physics and
materials science, for instance. The set $Z$ is called a {\em stationary 
Boolean model with spherical grains}. A simulated realization
for $d=2$ is shown in Figure \ref{fig1}.

It is a fundamental statistical problem to retrieve information on
$\G$ based on an observation of $Z$ in a bounded window $W$. 
Our aim in this paper is to propose and study a family of nonparametric
estimators of $\G$. The nonparametric 
estimation of the radius distribution $\G$ has been studied before; see \cite[Chapter 5.6]{Hall}, 
\cite{Molchanov90} or \cite{MolchanovStoyan}. In \cite{MolchanovStoyan}
a kernel estimator is obtained by the method of tangent points. The asymptotic
properties of this estimator are studied in \cite{HW00}. For earlier work on statistics for the Boolean model, we 
refer to \cite[Chapter 3.3]{SKM}, \cite{Molchanov} and the references therein.

In the following, we assume that all random elements are 
defined on an underlying probability space
$(\Omega,\mathcal{F},\BP)$. For a Borel set  
$A\subset\R^d$,  we write
$\Phi(A):=\card\{n\ge 1:\xi_n\in A\}$ and assume that $\Phi$ has a positive and finite intensity
\[
\gamma:=\BE\Phi([0,1]^d).
\]
Throughout the paper, let $B$ be a compact convex set which contains the origin $\o$ and 
a non-degenerate segment. We call $B$ {\em structuring element}
or {\em gauge body}, but we point out that $B$ need not be centrally symmetric 
or full-dimensional. The $B$-distance from a point $x\in\R^d$ to a set $A \subset \R^d$ is
\[
d_B(x,A) := \inf\{r \ge 0: (x+rB) \cap A\ne\emptyset\}\in[0,\infty].
\]
Clearly, if $\o\in\inter B$ and $A$ is nonempty and closed, then the infimum is a minimum. 
The most common structuring element is the unit ball $B(\o,1)$, for which we also write $B^d$ and which 
is based on a scalar product and a norm denoted by $\|\cdot\|$.  
For given $x \notin Z$, almost surely $d_B(x,Z)<\infty$ whenever $R$ satisfies $\BP(R>0)>0$. 
We always assume that this condition is fulfilled. Then 
almost surely there is a unique $n\in\N$ (that is, a ball $B(\xi_n,R_n)$) such that  
$(x+d_B(x,Z)B) \cap B(\xi_n,R_n)\ne\emptyset$ (see \cite[Lemma 3.1]{HLW} or
  \cite[Lemma 9.5.1]{SW08}). 
In this case, we define $r_B(x,Z)$ as $R_n$.
Figure \ref{fig1} illustrates the definition of $d_B(x,Z)$ and $r_B(x,Z)$ for $x \notin Z$ and $B=B^2$.

\begin{figure}[h]
  \centering
  \includegraphics[width=0.95\textwidth]{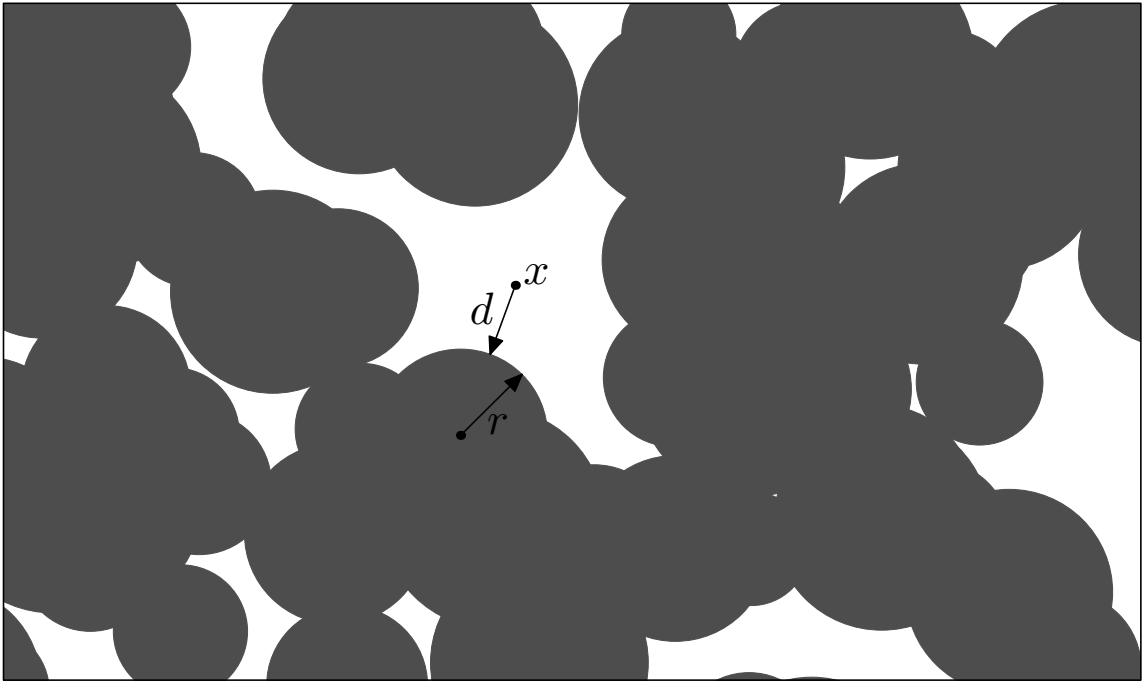}
  \caption{A simulated realization of a planar stationary Boolean 
  model $Z$ with spherical grains observed in a rectangular observation
  window. The symbol $d$ denotes $d_{B^2}(x,Z)$ and $r$ stands for
  $r_{B^2}(x,Z)$.}
  \label{fig1}
\end{figure}

For $s, r \ge 0$, we write $B_{s,r} := sB \oplus rB^d$ for the Minkowski sum of 
$sB$ and $rB^d$. Let $|A|_d$ denote the $d$-dimensional Lebesgue measure 
of a set $A\subset \R^d$, let   
$\kappa_k := |B^k|_k = {\pi^{k/2}}/{\Gamma(1+{k}/{2})}$  denote the volume of the $k$-dimensional
unit ball and write $V_j(B)$ for the $j$-th intrinsic volume of $B$ (see \cite[Chapter 14.3]{SW08}).
Then, for $t\in\R^+:=[0,\infty)$, the {\em empty space function} $F_B$ of $Z$ is given by
\begin{align} \notag
F_B(t) :=&\, \BP\bigl(d_B(\o,Z) \leq t\bigr) = \BP(Z \cap tB \ne \emptyset) 
= 1-\exp\left\{-\gamma \BE |B_{t,R}|_d\right\} \\ \label{1.3}
=&\, 1 - \exp\left\{-\gamma \sum_{j=0}^d \kappa_{d-j}V_j(B)t^j \BE R^{d-j}\right\}.
\end{align}
The empty space function is a useful summary statistics of random sets (see \cite{SKM,HBG1999}). 
In the case of a strictly convex gauge body $B$ a detailed study of $F_B$ for (non-stationary) 
germ-grain models can be found in \cite{HugLast}. 
We denote the complementary empty space function
by ${\bar F}_B(t) := 1 - F_B(t)$.
As a consequence of \cite[Theorem 3.2]{HLW}, we get 
for all measurable functions ${\tilde g}: [0,\infty] \times \R^+ \rightarrow \R^+$ such that 
${\tilde g}(0,r)={\tilde g}(\infty,r)=0$, $r\in\R^+$, 
and all $x\in\R^d$ that
\begin{align}\label{g}
\BE  {\tilde g}\bigl(d_B(x,Z),r_B(x,Z)\bigr) =
\gamma\int_0^\infty\int_0^\infty {\tilde g}(t,r)h_B(t,r){\bar F}_B(t)\,\d t\,\G(\d r)
\end{align}
with
\begin{align*} 
h_B(t,r) := \sum_{j=0}^{d-1} (j+1)\kappa_{d-1-j}V_{j+1}(B)r^{d-1-j}t^j,
\end{align*}
for $t,r\in [0,\infty)$; 
see also \cite[Theorem 9.5.2]{SW08}. Note that on the left-hand side of \eqref{g} the restriction 
to $\{0<d_B(x,Z)<\infty\}$ is expressed by the condition ${\tilde g}(0,r)={\tilde g}(\infty,r)=0$. 

For Borel sets $C\subset\R^+$, $A \subset\R^d$ and a measurable
function $f: [0,\infty] \rightarrow \R^+$ with $f(0)=f(\infty)=0$, we
define a random measure $\eta_A$ by
\begin{equation}\label{etaW}
\eta_A(C) := \int_A \I\{r_B(x,Z)\in C\} f\bigl(d_B(x,Z)\bigr)h_B\bigl(d_B(x,Z),r_B(x,Z)\bigr)^{-1}\,\d x,
\end{equation}
where $\I\{\cdot\}$ denotes the indicator function.
Here we put $0/0:=0$.  
Thus, in particular, the integration  effectively extends over the complement 
$Z^c:=\{x\in\R^d:d_B(x,Z)>0\}$ of $Z$. Throughout the paper we shall assume that 
\begin{equation}\label{beta}
0 < \beta := \int_0^\infty f(t){\bar F}_B(t)\,\d t < \infty.
\end{equation}
In view of \eqref{1.3} this is a rather weak assumption on $f$. 
Moreover, we assume that the origin is an interior point of $B$ if $\BP(R=0)>0$. This assumption ensures 
that $h_B(t,r)>0$ for $t\in (0,\infty)$ and $\mathbb{G}$-almost all $r\in\R^+$.  
By Fubini's theorem and \eqref{g}, we obtain
\begin{equation}\label{Eeta}
\BE \eta_A(C) = \gamma\,\beta\, |A|_d\, \G(C).
\end{equation}

Consider a compact convex observation window $W\subset\R^d$
with $|W|_d>0$. We propose an estimator $\widehat{\G}$
for $\G$ based on the information contained in the data
$\left\{\bigl(d_B(x,Z),r_B(x,Z)\bigr):x\in W\setminus Z\right\}$. 
Note that these data may also require
information from outside $W$. The estimator is given by
\begin{align}\label{estimator2}
\widehat{\G}(C) := \frac{\eta_W(C)}{\eta_W(\R^+)},
\end{align}
where $C \subset \R^+$ is a Borel set. 
If the denominator in \eqref{estimator2} is zero, then the numerator is zero as well, and we use 
the convention $0/0 := 0$. 
From \eqref{Eeta} we see that $\BE \eta_W(C) = \gamma\,\beta\, |W|_d\, \G(C)$
and $\BE \eta_W(\R^+) = \gamma\,\beta\,|W|_d$. 
This means that $\widehat{\G}$ is a {\em ratio-unbiased} estimator of $\G$.

The paper is organized as follows. In Section \ref{sec:second} we study
second order properties of \eqref{etaW}. We show that the asymptotic variance exists
and is given by a fairly explicit integral expression. Consequently, the estimator 
\eqref{estimator2} is asymptotically weakly consistent as the compact convex observation 
window $W$ is expanding. Strong consistency follows from the spatial ergodic theorem.
Section \ref{sec:normality} contains the proof of asymptotic normality 
under an integrability assumption on the function $f$.
In Section \ref{sec:planar} we consider the estimator $\widehat{\G}$ in the plane and 
for the spherical ($B=B(\o,1)$) as well as for the linear case ($B$ a segment). 
We also discuss some related estimators. 
A simulation study is performed to compare the behaviour
of different discrete versions of these estimators of the radius distribution $\G$.

\section{Second order properties} \label{sec:second}
\setcounter{equation}{0}

For a Borel set $A \subset \R^d$, we define the {\em restricted Boolean model} as
\[
Z(A) := \bigcup_{n: \xi_n \in A}  B(\xi_n,R_n).
\]
Clearly, $Z(A)$ is not stationary unless $A=\R^d$.  
Further, for  $ t \in \R^+$ the complementary empty space function of $Z(A)$ with respect to $x \in \R^d$ is defined by
\begin{align}
{\bar F}_B^A(x;t) &:= \BP\Bigl(d_B\bigl(x,Z(A)\bigr) > t\Bigr) = \BP\bigl((x+tB) \cap Z(A) = \emptyset\bigr) \nonumber \\
&= \BE \prod_{n \ge 1} \bigl(1-\I\{(x+tB) \cap B(\xi_n,R_n) \neq \emptyset\}\I\{\xi_n \in A\}\bigr) \nonumber \\
&= \exp\Bigl\{-\gamma\BE\int_{\R^d}\I\{(x+tB) \cap B(y,R) \not = \emptyset\}\I\{y \in A\}\,\d y\Bigr\} \nonumber \\
&= \exp\left\{-\gamma \BE \bigl|(x+B_{t,R}) \cap A\bigr|_d\right\}.
\label{empty1exp}
\end{align}
In particular, we have ${\bar F}_B^{\R^d}(x;t) = {\bar F}_B(t)$.

For Borel sets $A_1, A_2 \subset \R^d$ and $t_1,t_2 \in \R^+$, it will be convenient to introduce the complementary {\em second-order empty 
space function} with respect to $x_1,x_2 \in \R^d$ as
\begin{align}
&{\bar F}_B^{A_1,A_2}(x_1,x_2;t_1,t_2) := \BP\Bigl(d_B\bigl(x_1,Z(A_1)\bigr) > t_1, d_B\bigl(x_2,Z(A_2)\bigr) > t_2\Bigr) \label{empty2} \\
&\quad = \BP\bigl((x_1+t_1B) \cap Z(A_1) = \emptyset, (x_2+t_2B) \cap Z(A_2) = \emptyset\bigr) \nonumber \\
&\quad = \BE \prod_{n \ge 1} \bigl(1-\I\{(x_1+t_1B) \cap B(\xi_n,R_n) \not = \emptyset\}\I\{\xi_n \in A_1\}\bigr) \nonumber\\
&\quad\qquad\times \bigl(1-\I\{(x_2+t_2B) \cap B(\xi_n,R_n) \not = \emptyset\}
\I\{\xi_n \in A_2\}\bigr) \nonumber\\
&\quad = \exp\left\{-\gamma \BE \bigl|\bigl((x_1+B_{t_1,R}) \cap A_1\bigr) 
\cup \bigl((x_2+B_{t_2,R}) \cap A_2\bigr)\bigr|_d\right\}. \label{empty2exp}
\end{align}
This function is related to the second-order contact distribution function which
is studied in \cite{Ballani}.

In order to obtain a more concise statement in the subsequent Lemma \ref{l1} (and again in the proof of Theorem \ref{clt1}), 
we introduce for given Borel sets $A_1, A_2 \subset \R^d$ two functions, $I_1(A_1,A_2)$ 
and $I_2(A_1,A_2)$, depending on the arguments $(x_1,x_2,y,r)\in(\R^d)^3\times\R^+$ and 
$(x_1,x_2,y_1,y_2,r_1,r_2)\in(\R^d)^4\times(\R^+)^2 $, respectively, which are defined by
\[
I_1(A_1,A_2) (x_1,x_2,y,r):= \I\{y \in A_1 \cap A_2\}{\bar F}_B^{A_1,A_2}\Bigl(x_1,x_2;
d_B\bigl(x_1,B(y,r)\bigr),d_B\bigl(x_2,B(y,r)\bigr)\Bigr)
\]
and
\begin{align*}
&I_2(A_1,A_2)(x_1,x_2,y_1,y_2,r_1,r_2) := \I\{y_1 \in A_1\}\I\{y_2 \in A_2\}\\
&\qquad \times
\Bigl[
\Bigl(1-\I\{y_2 \in A_1\}\I\left\{d_B\bigl(x_1,B(y_2,r_2)\bigr) \leq d_B\bigl(x_1,B(y_1,r_1)\bigr)\right\}\Bigr) \\
&\qquad\qquad\times \Bigl(1-\I\{y_1 \in A_2\}\I\left\{d_B\bigl(x_2,B(y_1,r_1)\bigr) \leq d_B\bigl(x_2,B(y_2,r_2)\bigr)\right\}\Bigr) \\
&\qquad\qquad\times {\bar F}_B^{A_1,A_2}\Bigl(x_1,x_2;d_B\bigl(x_1,B(y_1,r_1)\bigr),d_B\bigl(x_2,B(y_2,r_2)\bigr)\Bigr) \\
&\qquad\qquad\qquad - {\bar F}_B^{A_1}\Bigl(x_1;d_B\bigl(x_1,B(y_1,r_1)\bigr)\Bigr)
{\bar F}_B^{A_2}\Bigl(x_2;d_B\bigl(x_2,B(y_2,r_2)\bigr)\Bigr)\Bigr].
\end{align*}
If the arguments of these two functions are clear from the context, they are sometimes omitted.

\begin{lemma} \label{l1}
Let $A_1, A_2 \subset \R^d$ be Borel sets and let $x_1, x_2 \in \R^d$. 
If $\tilde{g}: [0,\infty] \times \R^+ \rightarrow \R^+$ is a measurable function with $\tilde g(0,r)=\tilde g(\infty,r)=0$ for $r\in\R^+$, then
\begin{align*}
&\cov \Bigl(\tilde g\bigl(d_B(x_1,Z(A_1)),r_B(x_1,Z(A_1))\bigr),\tilde g\bigl(d_B(x_2,Z(A_2)),r_B(x_2,Z(A_2))\bigr)\Bigl)\\
&\quad = \gamma \int_0^\infty \int_{\R^d} \tilde g\bigl(d_B(x_1,B(y,r)),r\bigr)
\tilde g\bigl(d_B(x_2,B(y,r)),r\bigr) I_1(A_1,A_2) (x_1,x_2,y,r) 
\,\d y\,\G(\d r) \\
&\qquad + \gamma^2 \int_0^\infty\int_0^\infty \int_{\R^d}\int_{\R^d} \tilde g\bigl(d_B(x_1,B(y_1,r_1)),r_1\bigr)
\tilde g\bigl(d_B(x_2,B(y_2,r_2)),r_2\bigr)  \\
&\qquad\qquad\quad\times I_2(A_1,A_2)(x_1,x_2,y_1,y_2,r_1,r_2)\,\d y_1\,\d y_2\,\G(\d r_1)\,\G(\d r_2).
\end{align*}
\end{lemma}

\begin{proof}
For $n \in \N$, $x\in\R^d$, and $i\in\{1,2\}$, we define the event
\[
D_n^{(i)}(x) := \Bigl\{d_B\bigl(x,\bigcup_{k \ne n: \xi_k \in A_i} B(\xi_k,R_k)\bigr) >
d_B\bigl(x,B(\xi_n,R_n)\bigr)\Bigr\}.
\]
Then 
\begin{align*}
&\BE \tilde g\bigl(d_B(x_1,Z(A_1)),r_B(x_1,Z(A_1))\bigr)\tilde g\bigl(d_B(x_2,Z(A_2)),r_B(x_2,Z(A_2))\bigr) \\
&\quad = \BE \sum_{n: \xi_n \in A_1} \sum_{m: \xi_m \in A_2} \I_{D_n^{(1)}(x_1) \cap D^{(2)}_m(x_2)}
\tilde g\bigl(d_B(x_1,B(\xi_n,R_n)),R_n\bigr)\tilde g\bigl(d_B(x_2,B(\xi_m,R_m)),R_m\bigr) \\
&\quad = \BE \sum_{n: \xi_n \in A_1 \cap A_2} \I_{D_n^{(1)}(x_1) \cap D^{(2)}_n(x_2)}
\tilde g\bigl(d_B(x_1,B(\xi_n,R_n)),R_n\bigr)\tilde g\bigl(d_B(x_2,B(\xi_n,R_n)),R_n\bigr) \\
&\qquad + \BE \dsum_{n \ne m: \xi_n \in A_1, \xi_m \in A_2} \I_{D_n^{(1)}(x_1) \cap D^{(2)}_m(x_2)}
\tilde g\bigl(d_B(x_1,B(\xi_n,R_n)),R_n\bigr)\tilde g\bigl(d_B(x_2,B(\xi_m,R_m)),R_m\bigr) \\
&\quad =: J_1 + J_2.
\end{align*}
Applying Mecke's formula (see \cite[Corollary 3.2.3]{SW08}), we obtain
\begin{align*}
J_1 &= \gamma \int_0^\infty \int_{A_1 \cap A_2} \tilde g\bigl(d_B(x_1,B(y,r)),r\bigr)
\tilde g\bigl(d_B(x_2,B(y,r)),r\bigr)  \\
&\qquad \times \BP\Bigl(d_B\bigl(x_1,Z(A_1)\bigr)>d_B\bigl(x_1,B(y,r)\bigr),
d_B\bigl(x_2,Z(A_2)\bigr)>d_B\bigl(x_2,B(y,r)\bigr)\Bigr)
\,\d y\,\G(\d r)
\end{align*}
and
\begin{align*}
J_2 &= \gamma^2 \int_0^\infty\int_0^\infty \int_{A_2}\int_{A_1} \tilde g\bigl(d_B(x_1,B(y_1,r_1)),r_1\bigr)
\tilde g\bigl(d_B(x_2,B(y_2,r_2)),r_2\bigr)  \\
&\qquad\times \BE \I\left\{d_B\bigl(x_1,Z_{y_2}(A_1)\bigr)>d_B\bigl(x_1,B(y_1,r_1)\bigr)\right\} 
\I\left\{d_B\bigl(x_2,Z_{y_1}(A_2)\bigr)>d_B\bigl(x_2,B(y_2,r_2)\bigr)\right\} \\
&\qquad\times \,\d y_1\,\d y_2\,\G(\d r_1)\,\G(\d r_2),
\end{align*}
where $Z_{y_2}(A_1) = Z(A_1) \cup B(y_2,r_2)$ if $y_2 \in A_1$ and $Z_{y_2}(A_1) = Z(A_1)$ if
$y_2 \notin A_1$. Analogously, $Z_{y_1}(A_2) = Z(A_2) \cup B(y_1,r_1)$ if $y_1 \in A_2$ and
$Z_{y_1}(A_2) = Z(A_2)$ if $y_1 \notin A_2$.
Hence,
\begin{align*}
J_2 &= \gamma^2 \int_0^\infty\int_0^\infty \int_{A_2}\int_{A_1} \tilde g\bigl(d_B(x_1,B(y_1,r_1)),r_1\bigr)
\tilde g\bigl(d_B(x_2,B(y_2,r_2)),r_2\bigr)  \\
&\qquad\times \Bigl(1-\I\{y_2 \in A_1\}\I\{d_B\bigl(x_1,B(y_2,r_2)\bigr) \leq d_B\bigl(x_1,B(y_1,r_1)\bigr)\}\Bigr) \\
&\qquad\times \Bigl(1-\I\{y_1 \in A_2\}\I\{d_B\bigl(x_2,B(y_1,r_1)\bigr) \leq d_B\bigl(x_2,B(y_2,r_2)\bigr)\}\Bigr) \\
&\qquad\times \BP\Bigl(d_B\bigl(x_1,Z(A_1)\bigr) > d_B\bigl(x_1,B(y_1,r_1)\bigr),
d_B\bigl(x_2,Z(A_2)\bigr)>d_B\bigl(x_2,B(y_2,r_2)\bigr)\Bigr) \\
&\qquad\times\,\d y_1\,\d y_2\,\G(\d r_1)\,\G(\d r_2).
\end{align*}
Finally, 
\begin{align*}
&\BE \tilde g\bigl(d_B(x_1,Z(A_1)),r_B(x_1,Z(A_1))\bigr) = 
\BE \sum_{n: \xi_n \in A_1} \I_{D_n^{(1)}(x_1)}\tilde g\bigl(d_B(x_1,B(\xi_n,R_n)),R_n\bigr) \\
&\qquad = \gamma\int_0^\infty \int_{A_1} 
\tilde g\bigl(d_B(x_1,B(y_1,r_1)),r_1\bigr){\bar F}_B^{A_1}\bigl(x_1;d_B(x_1,B(y_1,r_1))\bigr)\,\d y_1\,\G(\d r_1).
\end{align*}
\end{proof}

Our aim is to analyze the second-order properties of the random measure $\eta_A$ given by \eqref{etaW}. 
For this reason, we work with the complementary second-order empty space function \eqref{empty2}. 
For $A_1=A_2=\R^d$, $t_1,t_2\in \R^+$, and $u=x_2-x_1$, by the stationarity of $Z$ this function turns into
\begin{align}
{\bar F}_B^{(2)}(u;t_1,t_2) :=\,& \BP\bigl(d_B(\o,Z) > t_1, d_B(u,Z) > t_2\bigr) \nonumber\\
=\,& \exp\Bigl\{-\gamma \BE \bigl|B_{t_1,R} \cup (u+B_{t_2,R})\bigr|_d\Bigr\} \nonumber\\
=\,& {\bar F}_B(t_1){\bar F}_B(t_2)\exp\left\{\gamma \BE \kappa_B(u;t_1,t_2,R)\right\},\label{vor2.6}
\end{align}
where 
\begin{equation} \label{kappaB}
\kappa_B(u;t_1,t_2,r) := \bigl|B_{t_1,r} \cap (u+B_{t_2,r})\bigr|_d.
\end{equation}
Observe that for any $u \in \R^d$ and $t_1, t_2 \in \R^+$, we have
\begin{equation}\label{FB2lowerestimate}
{\bar F}_B^{(2)}(u;t_1,t_2) \geq {\bar F}_B(t_1){\bar F}_B(t_2)
\end{equation}
and
\begin{equation}
{\bar F}_B^{(2)}(u;t_1,t_2) \leq \exp\Bigl\{-\frac{\gamma}{2}\BE \bigl(|B_{t_1,R}|_d 
+ |B_{t_2,R}|_d\bigr)\Bigr\} = \sqrt{{\bar F}_B(t_1){\bar F}_B(t_2)}.
\label{barF}
\end{equation}
These inequalities will be used subsequently. 
In addition,  we shall need the assumption
\begin{equation}
\int_0^\infty f(t)\e^{-ct}\,\d t < \infty ,
\label{assumption}
\end{equation}
where $c := 4^{-1}\gamma\kappa_{d-1} V_1(B)\mathbb{E} R^{d-1}<\infty$ and $c>0$ since $V_1(B)>0$ (recall 
that $B$ contains a non-degenerate line segment) and $\BP(R>0)>0$.

\begin{proposition}\label{p1}
Assume that \eqref{assumption} is satisfied. 
If $C \subset \R^+$ is a Borel set and  $W_1,W_2\subset\R^d$ are compact convex sets, then
\[
\cov\bigl(\eta_{W_1}(C),\eta_{W_2}(C)\bigr) = 
\int_{\R^d} |W_1 \cap (W_2-u)|_d \\
\left[\gamma \tau_1(C,u)+ \gamma^2 \tau_2(C,u)\right]\,\d u,
\]
where
\begin{align}
\tau_1(C,u) :=\,& \int_C \int_{\R^d} \frac{f\bigl(d_B(x,B(\o,r))\bigr)}{h_B\bigl(d_B(x,B(\o,r)),r\bigr)}
\frac{f\bigl(d_B(u+x,B(\o,r))\bigr)}{h_B\bigl(d_B(u+x,B(\o,r)),r\bigr)} \label{label1}\\
\,&\qquad\qquad \times {\bar F}_B^{(2)}\bigl(u;d_B(x,B(\o,r)),d_B(u+x,B(\o,r))\bigr)\,\nonumber
\d x\,\G(\d r)
\end{align}
and
\begin{align}
\tau_2(C,u) :=& \int_C\int_C \int_{\R^d} \int_{\R^d}
\frac{f\bigl(d_B(x_1,B(\o,r_1))\bigr)}{h_B\bigl(d_B(x_1,B(\o,r_1)),r_1\bigr)}
\frac{f\bigl(d_B(x_2,B(\o,r_2))\bigr)}{h_B\bigl(d_B(x_2,B(\o,r_2)),r_2\bigr)} \label{label2}\\
\,&\qquad\qquad \times q(u;x_1,x_2,r_1,r_2)\,\d x_1\,\d x_2\,\G(\d r_1)\,\G(\d r_2),\nonumber
\end{align}
for $u\in\R^d$, and
\begin{align*}
&q(u;x_1,x_2,r_1,r_2)\\
& := \I\bigl\{d_B(x_2,B(u,r_2))>d_B(x_1,B(\o,r_1))\bigr\}
\I\bigl\{d_B(x_1,B(-u,r_1))>d_B(x_2,B(\o,r_2))\bigr\} \\
&\qquad \times {\bar F}_B^{(2)}\bigl(u;d_B(x_1,B(\o,r_1)),d_B(x_2,B(\o,r_2))\bigr)
- {\bar F}_B\bigl(d_B(x_1,B(\o,r_1))\bigr){\bar F}_B\bigl(d_B(x_2,B(\o,r_2))\bigr), 
\end{align*}
for $x_1,x_2\in\R^d$ and $r_1,r_2\in\R^+$. 
\end{proposition}

\begin{proof}
To abbreviate the notation, we define the function 
\begin{align}\label{shortg}
g(t,r):=\I\{r \in C\} f(t)h_B(t,r)^{-1},
\end{align}
for $t\in[0,\infty]$ and $r\in\R^+$, with the previous conventions in the cases where $t\in \{0,\infty\}$. 
Recall also that $h_B(t,r)>0$ for $t\in (0,\infty)$ and $\mathbb{G}$-almost all $r\in\R^+$. 
Using Fubini's theorem and stationarity, we get
\begin{align*}
&\cov \bigl(\eta_{W_1}(C),\eta_{W_2}(C)\bigr) \\
&\quad = \int_{W_1}\int_{W_2}
\cov \Bigl(g\bigl(d_B(x_1,Z),r_B(x_1,Z)\bigr),g\bigl(d_B(x_2,Z),r_B(x_2,Z)\bigr)\Bigr)\,\d x_2\,\d x_1 \\
&\quad = \int_{\R^d} |W_1 \cap (W_2-u)|_d 
\cov \Bigl(g\bigl(d_B(\o,Z),r_B(\o,Z)\bigr),g\bigl(d_B(u,Z),r_B(u,Z)\bigr)\Bigr)\,\d u. 
\end{align*}
By Lemma \ref{l1} with $A_1=A_2=\R^d$, $x_1=\o$ and $x_2=u$, we obtain that
\[
\cov \Bigl(g\bigl(d_B(\o,Z),r_B(\o,Z)\bigr),g\bigl(d_B(u,Z),r_B(u,Z)\bigr)\Bigr) = J_1(u) + J_{21}(u) - J_{22},
\]
where
\begin{align*}
J_1(u) &:= \gamma \int_0^\infty \int_{\R^d} g\bigl(d_B(\o,B(x,r)),r\bigr)
g\bigl(d_B(u,B(x,r)),r\bigr)  \\
&\qquad\qquad \times {\bar F}_B^{(2)}\bigl(u;d_B(\o,B(x,r)),d_B(u,B(x,r))\bigr)\,\d x\,\G(\d r), 
\end{align*}
\begin{align*}
J_{21}(u) &:= \gamma^2 \int_0^\infty\int_0^\infty \int_{\R^d} \int_{\R^d}
g\bigl(d_B(\o,B(x_1,r_1)),r_1\bigr)g\bigl(d_B(\o,B(x_2,r_2)),r_2\bigr) \\
&\qquad\quad \times \I\bigl\{d_B(-u,B(x_2,r_2))>d_B(\o,B(x_1,r_1))\bigr\}\\
&\qquad\quad \times \I\bigl\{d_B(u,B(x_1,r_1))>d_B(\o,B(x_2,r_2))\bigr\} \\
&\qquad\quad \times {\bar F}_B^{(2)}\bigl(u;d_B(\o,B(x_1,r_1)),d_B(\o,B(x_2,r_2))\bigr)
\,\d x_1\,\d x_2\,\G(\d r_1)\,\G(\d r_2),
\end{align*} 
and
\begin{align*}
J_{22} &:= \gamma^2 \int_0^\infty\int_0^\infty \int_{\R^d} \int_{\R^d}
g\bigl(d_B(\o,B(x_1,r_1)),r_1\bigr)g\bigl(d_B(\o,B(x_2,r_2)),r_2\bigr) \\
&\qquad\quad \times {\bar F}_B\bigl(d_B(\o,B(x_1,r_1))\bigr){\bar F}_B\bigl(d_B(\o,B(x_2,r_2))\bigr)\,\d x_1\,\d x_2\,\G(\d r_1)\,\G(\d r_2).
\end{align*}
Using $d_B\bigl(u,B(x,r)\bigr) = d_B\bigl(u-x,B(\o,r)\bigr)$ and the reflection invariance of Lebesgue measure,
we deduce that
\begin{align*}
J_1(u) &= \gamma \int_0^\infty \int_{\R^d} g\bigl(d_B(x,B(\o,r)),r\bigr)
g\bigl(d_B(u+x,B(\o,r)),r\bigr)  \\
&\qquad\quad \times {\bar F}_B^{(2)}\bigl(u;d_B(x,B(\o,r)),d_B(u+x,B(\o,r))\bigr)\,\d x\,\G(\d r), 
\end{align*}
\begin{align*}
J_{21}(u) &= \gamma^2 \int_0^\infty \int_0^\infty \int_{\R^d} \int_{\R^d}
g\bigl(d_B(x_1,B(\o,r_1)),r_1\bigr)g\bigl(d_B(x_2,B(\o,r_2)),r_2\bigr) \\
&\qquad\quad \times \I\bigl\{d_B(x_2,B(u,r_2))>d_B(x_1,B(\o,r_1))\bigr\}\\
&\qquad\quad \times\I\bigl\{d_B(x_1,B(-u,r_1))>d_B(x_2,B(\o,r_2))\bigr\} \\
&\qquad\quad\times {\bar F}_B^{(2)}\bigl(u;d_B(x_1,B(\o,r_1)),d_B(x_2,B(\o,r_2))\bigr)
\,\d x_1\,\d x_2\,\G(\d r_1)\,\G(\d r_2), 
\end{align*}
and
\begin{align*}
J_{22} &= \gamma^2 \int_0^\infty\int_0^\infty \int_{\R^d} \int_{\R^d}
g\bigl(d_B(x_1,B(\o,r_1)),r_1\bigr)g\bigl(d_B(x_2,B(\o,r_2)),r_2\bigr) \\
&\qquad \quad \times {\bar F}_B\bigl(d_B(x_1,B(\o,r_1))\bigr){\bar F}_B\bigl(d_B(x_2,B(\o,r_2))\bigr)
\,\d x_1\,\d x_2\,\G(\d r_1)\,\G(\d r_2).
\end{align*}
The assertion now follows by recalling \eqref{shortg}. The integrability of $\tau_1(C,\cdot)$ and $\tau_2(C,\cdot)$, which  
is explicitly stated in \eqref{integrab2.4}, 
will be shown in the proof of Theorem \ref{asvar} and is implied by the assumption \eqref{assumption}.
\end{proof}

\begin{remark} \label{Bball}\rm
Recall that $\|\cdot\|$ denotes the Euclidean norm on $\R^d$. If $B=B^d$ is the unit ball, then $d_{B^d}\bigl(u,B(x,r)\bigr) = (\|x-u\|-r)^+$,
\[
h_{B^d}(t,r) = \sum_{j=0}^{d-1} d\kappa_d\binom{d-1}{j} r^{d-1-j}t^j = d\kappa_d (t+r)^{d-1},
\]
and
\[
\kappa_{B^d}(u;t_1,t_2,r) = \bigl|B(\o,t_1+r) \cap B(u,t_2+r)\bigr|_d.
\]
Hence, $\tau_1(C,u)$ and $\tau_2(C,u)$ from Proposition \ref{p1} 
may be slightly simplified. In particular, then we have
\begin{align*}
\tau_2(C,u) &= \int_C\int_C \int_{\R^d} \int_{\R^d}
\frac{f\bigl((\|x_1\|-r_1)^+\bigr)}{h_{B^d}\bigl((\|x_1\|-r_1)^+,r_1\bigr)}
\frac{f\bigl((\|x_2\|-r_2)^+\bigr)}{h_{B^d}\bigl((\|x_2\|-r_2)^+,r_2\bigr)} \\
&\qquad \times \Bigl[\I\left\{(\|x_2-u\|-r_2)^+ > (\|x_1\|-r_1)^+\right\}
\I\left\{(\|x_1+u\|-r_1)^+ > (\|x_2\|-r_2)^+\right\} \\
&\qquad \times {\bar F}_{B^d}^{(2)}\bigl(u;(\|x_1\|-r_1)^+,(\|x_2\|-r_2)^+\bigr)
-{\bar F}_{B^d}\bigl((\|x_1\|-r_1)^+\bigr){\bar F}_{B^d}\bigl((\|x_2\|-r_2)^+\bigr)\Bigr] \\
&\qquad\times \,\d x_1\,\d x_2\,\G(\d r_1)\,\G(\d r_2) \\
&= \int_C \int_C \int_{0}^\infty \int_{\BS^{d-1}} \int_{0}^\infty \int_{\BS^{d-1}}
\frac{f(s_1)}{h_{B^d}(s_1,r_1)}(s_1+r_1)^{d-1}\frac{f(s_2)}{h_{B^d}(s_2,r_2)}(s_2+r_2)^{d-1} \\
&\qquad \times \Bigl[\I\left\{(\|(s_2+r_2)v_2-u\|-r_2)^+ > s_1\right\}\I\left\{(\|(s_1+r_1)v_1+u\|-r_1)^+ > s_2\right\} \\
&\qquad \times {\bar F}_{B^d}^{(2)}(u;s_1,s_2)-{\bar F}_{B^d}(s_1){\bar F}_{B^d}(s_2)\Bigr]
\,\H^{d-1}(\d v_1)\,\d s_1\,\H^{d-1}(\d v_2)\,\d s_2\,\G(\d r_1)\,\G(\d r_2) \\
&= \int_C \int_C \int_0^\infty \int_0^\infty f(s_1)f(s_2)
\Bigl[\frac{\H^{d-1}\bigl(\partial B(\o,s_2+r_2) \cap B(u,s_1+r_2)^c\bigr)}{\H^{d-1}\bigl(\partial B(\o,s_2+r_2)\bigr)} \\
&\qquad \times \frac{\H^{d-1}\bigl(\partial B(\o,s_1+r_1) \cap B(-u,s_2+r_1)^c\bigr)}{\H^{d-1}\bigl(\partial B(\o,s_1+r_1)\bigr)}
{\bar F}_{B^d}^{(2)}(u;s_1,s_2)-{\bar F}_{B^d}(s_1){\bar F}_{B^d}(s_2)\Bigr] \\
&\qquad\times\,\d s_1\,\d s_2\,\G(\d r_1)\,\G(\d r_2),
\end{align*}
where $\BS^{d-1}$ is the unit sphere in $\R^d$, 
$\H^{d-1}$ is the $(d-1)$-dimensional Hausdorff measure, and $\partial B(x,r)$ is the boundary 
of $B(x,r)$. We used that $f\bigl((\|x\|-r)^+\bigr)$ is non-zero only if $\|x\| > r$. Then $x = (s+r)v$ for $s>0$ and 
$v \in \BS^{d-1}$.
\end{remark}

Next we state a special case of \cite[Theorem 2.1 and Remark 3.1]{HLW}  in the form 
needed in the present context. Let ${\tilde g}:\R^d\to[0,\infty]$ be measurable, and let $K,B\subset\R^d$ be convex bodies  
such that $\o\in B$ and $K, B$ are 
in general relative position. Since in our application, we shall only need the case $K=rB^d$, for $r\in \R^+$, the assumption of general relative position  will be satisfied for any choice of $B$. Then we have
\begin{align*} 
&\int_{\R^d}\I\{0< d_B(z,rB^d)<\infty\}{\tilde g}(z)\, \d z\\
&\qquad =\sum_{j=0}^{d-1}\binom{d-1}{j}
\int_0^\infty\int t^{d-1-j}{\tilde g}(z+tb)\, \Theta_{j;d-j}\bigl(rB^d;B^*;\d(z,b)\bigr)\, \d t,
\end{align*}
where $B^*:=-B$ and the mixed support measures $\Theta_{j;d-j}(rB^d;B^*;\cdot)$, $j\in\{0,\ldots,d-1\}$, are finite Borel measures on $\R^{2d}$. 
Using \cite[(14.18)]{SW08} (cf.~\cite[(4.2.26) and (5.3.8)]{Sch93}) and \cite[middle of p.~327]{KW}, we obtain for the total measures 
$\Theta_{j;d-j}(rB^d;B^*;\R^{2d})=r^jd\binom{d}{j}\kappa_jV_{d-j}(B)$. In particular, this yields for any measurable function $\tilde f:[0,\infty]\to[0,\infty]$ with $\tilde f(0)=\tilde f(\infty)=0$ that 
\begin{equation}\label{refB}
\int_{\R^d}\tilde f\bigl(d_B(z,rB^d)\bigr)\,\d z=\int_{0}^\infty h_B(t,r)\tilde f(t)\, \d t.
\end{equation}

We now turn to the asymptotic properties of the ratio-unbiased
estimator \eqref{estimator2}. Our setting is similar to \cite{Molchanov},
where all limit theorems refer to a growing observation window in $\R^d$. 
More formally, we consider a  
sequence $(W_n)_{n\in\N}$ of compact, convex sets $W_n\subset\R^d$ such that 
$W_n \subset W_{n+1}$ for all $n \in \N$ and the inradius of $W_n$
tends to $\infty$ as $n \to \infty$.

\begin{theorem} \label{asvar} 
Assume that \eqref{assumption} is fulfilled. Then
\begin{equation}
\frac{\var \eta_{W_n}(C)}{|W_n|_d} \longn \sigma^2(C).
\label{asvareta}
\end{equation}
The asymptotic variance is finite and given by
\begin{equation}
\sigma^2(C) = \gamma \int_{\R^d} \tau_1(C,u)\,\d u + \gamma^2 \int_{\R^d} \tau_2(C,u)\,\d u,
\label{asvarC}
\end{equation}
where $\tau_1(C,u)$ and $\tau_2(C,u)$ are defined in \eqref{label1} and \eqref{label2}, respectively. 
Moreover, if $0<\G(C)<1$, then $\sigma^2(C)>0$.
\end{theorem}

\begin{proof}
Suppose that $x\in W_n$ and $u\in\R^d$. If  $x+u\notin W_n$, then $d_{B^d}(x,\partial W_n)\le \|u\|$. Hence, we obtain
$$
|W_n|_d-|\{x\in W_n:d_{B^d}(x,\partial W_n)\le \|u\|\}|_d\le |W_n \cap (W_n-u)|_d\le |W_n|_d.
$$
Thus, \cite[Lemma 10.15 (ii)]{Kallenberg} implies that
\[
\frac{|W_n \cap (W_n-u)|_d}{|W_n|_d} \longn 1 \quad \hbox{for any $u \in \R^d$}.
\]
Therefore Lebesgue's dominated convergence theorem and 
Proposition \ref{p1} yield \eqref{asvareta} provided that
\begin{equation}\label{integrab2.4}
\int_{\R^d} \tau_1(C,u)\,\d u < \infty \quad \hbox{and} \quad
\int_{\R^d} |\tau_2(C,u)|\,\d u < \infty.
\end{equation}

Using \eqref{barF} we have
\begin{align*}
\int_{\R^d} \tau_1(C,u)\,\d u &\leq \int_C \int_{\R^d}\int_{\R^d} 
\frac{f\bigl(d_B(x,rB^d)\bigr)}{h_B\bigl(d_B(x,rB^d),r\bigr)}
\frac{f\bigl(d_B(y,rB^d)\bigr)}{h_B\bigl(d_B(y,rB^d),r\bigr)} \\
&\qquad \times \sqrt{{\bar F}_B\bigl(d_B(x,rB^d)\bigr){\bar F}_B\bigl(d_B(y,rB^d)\bigr)}
\,\d x\,\d y\,\G(\d r) \\
&= \int_C \left(\int_{\R^d} 
\frac{f\bigl(d_B(x,rB^d)\bigr)}{h_B\bigl(d_B(x,rB^d),r\bigr)}
\sqrt{{\bar F}_B\bigl(d_B(x,rB^d)\bigr)}\,\d x\right)^2\,\G(\d r).
\end{align*}
An application of \eqref{refB} shows that
$$
\int_{\R^d}\frac{f\bigl(d_B(x,rB^d)\bigr)}{h_B\bigl(d_B(x,rB^d),r\bigr)}
\sqrt{{\bar F}_B\bigl(d_B(x,rB^d)\bigr)}\, \d x=\int_0^\infty f(t)\sqrt{{\bar F}_B(t)}\, \d t
$$
and thus we obtain
\begin{align*}
\int_{\R^d} \tau_1(C,u)\,\d u &\leq \int_C \left( \int_0^\infty f(t)\sqrt{{\bar F}_B(t)}\,\d t\right)^2\,\G(\d r) \\
&\leq \G(C) \left( \int_0^\infty f(t)\e^{-2ct}\,\d t\right)^2 < \infty,
\end{align*}
where we use that ${\bar F}_B(t)\le \exp\{-4ct\}$ and assumption \eqref{assumption}.

In order to show that 
$$\int_{\R^d} |\tau_2(C,u)|\,\d u < \infty,$$
we first rewrite $q(u;x_1,x_2,r_1,r_2)$ as the difference of two non-negative terms, that is, $q=q_1-q_2$ with
\begin{align*}
q_1(u;x_1,x_2,r_1,r_2) &:= {\bar F}_B^{(2)}\bigl(u;d_B(x_1,B(\o,r_1)),d_B(x_2,B(\o,r_2))\bigr) \\
&\qquad - {\bar F}_B\bigl(d_B(x_1,B(\o,r_1))\bigr){\bar F}_B\bigl(d_B(x_2,B(\o,r_2))\bigr)  ,
\end{align*}
which is non-negative by \eqref{FB2lowerestimate}, and 
\begin{align*}
&q_2(u;x_1,x_2,r_1,r_2) := {\bar F}_B^{(2)}\bigl(u;d_B(x_1,B(\o,r_1)),d_B(x_2,B(\o,r_2))\bigr) \\
&\qquad \times \Bigl(1-\I\bigl\{d_B(x_2,B(u,r_2))>d_B(x_1,B(\o,r_1))\bigr\}
\I\bigl\{d_B(x_1,B(-u,r_1))>d_B(x_2,B(\o,r_2))\bigr\}\Bigr) ,
\end{align*}
for $u,x_1,x_2\in\R^d$ and $r_1,r_2\in\R^+$. 
Using \eqref{vor2.6}, \eqref{barF} and the inequality $1-\e^{-a} \le a$, for $a \ge 0$, we get
\begin{align*}
q_1(u;x_1,x_2,r_1,r_2) &\le  {\bar F}_B^{(2)}\bigl(u;d_B(x_1,B(\o,r_1)),d_B(x_2,B(\o,r_2))\bigr)\\
&\qquad \times\Bigl(
1-\exp\bigl\{-\gamma\BE \kappa_B(u;d_B(x_1,B(\o,r_1)),d_B(x_2,B(\o,r_2)),R)\bigr\}\Bigr)\\
&\leq \sqrt{{\bar F}_B\bigl(d_B(x_1,B(\o,r_1))\bigr){\bar F}_B\bigl(d_B(x_2,B(\o,r_2))\bigr)} \\
&\qquad \times \gamma\BE \kappa_B\bigl(u;d_B(x_1,B(\o,r_1)),d_B(x_2,B(\o,r_2)),R\bigr).
\end{align*}
Moreover, the inequality $1-(1-a)(1-b) \leq a+b$, for $a,b \ge 0$, and again \eqref{barF} imply that 
\begin{align*}
q_2(u;x_1,x_2,r_1,r_2) &\le \sqrt{{\bar F}_B\bigl(d_B(x_1,B(\o,r_1))\bigr){\bar F}_B\bigl(d_B(x_2,B(\o,r_2))\bigr)} \\
&\qquad \times \Bigl(\I\bigl\{d_B(x_2,B(u,r_2)) \le d_B(x_1,B(\o,r_1))\bigr\} \\
&\qquad\qquad + \I\bigl\{d_B(x_1,B(-u,r_1)) \le d_B(x_2,B(\o,r_2))\bigr\}\Bigr).
\end{align*}
Combining these bounds, we arrive at
\begin{align*}
\int_{\R^d} |\tau_2(C,u)|\,\d u &\leq \int_C\int_C\int_{\R^d}\int_{\R^d}\int_{\R^d}
\frac{f\bigl(d_B(x_1,B(\o,r_1))\bigr)}{h_B\bigl(d_B(x_1,B(\o,r_1)),r_1\bigr)}
\frac{f\bigl(d_B(x_2,B(\o,r_2))\bigr)}{h_B\bigl(d_B(x_2,B(\o,r_2)),r_2\bigr)} \\
&\qquad \times \sqrt{{\bar F}_B\bigl(d_B(x_1,B(\o,r_1))\bigr){\bar F}_B\bigl(d_B(x_2,B(\o,r_2))\bigr)} \\
&\qquad \times \Bigl[\gamma\BE \kappa_B\bigl(u;d_B(x_1,B(\o,r_1)),d_B(x_2,B(\o,r_2)),R\bigr) \\
&\qquad\qquad + \I\bigl\{d_B(x_2,B(u,r_2)) \le d_B(x_1,B(\o,r_1))\bigr\}\\
&\qquad\qquad + \I\bigl\{d_B(x_1,B(-u,r_1)) \le d_B(x_2,B(\o,r_2))\bigr\}\Bigr] \\
&\qquad \times \,\d x_1\, \d x_2\,\d u\,\G(\d r_1)\,\G(\d r_2).
\end{align*}
The preceding expression splits naturally into three summands which will be bounded from above separately. For the first bound, we observe that  
by Fubini's theorem 
\[
\BE \int_{\R^d} \kappa_B(u;s_1,s_2,R)\,\d u = \BE |B_{s_1,R}|_d |B_{s_2,R}|_d.
\]
Then we apply \eqref{refB} to get
\begin{align*}
&\int_C\int_C \int_{\R^d}\int_{\R^d}\int_{\R^d} 
\frac{f\bigl(d_B(x_1,B(\o,r_1))\bigr)}{h_B\bigl(d_B(x_1,B(\o,r_1)),r_1\bigr)}
\frac{f\bigl(d_B(x_2,B(\o,r_2))\bigr)}{h_B\bigl(d_B(x_2,B(\o,r_2)),r_2\bigr)} \\
&\qquad \times \sqrt{{\bar F}_B\bigl(d_B(x_1,B(\o,r_1))\bigr){\bar F}_B\bigl(d_B(x_2,B(\o,r_2))\bigr)} \\
&\qquad \times \gamma\BE \kappa_B\bigl(u;d_B(x_1,B(\o,r_1)),d_B(x_2,B(\o,r_2)),R\bigr) \\
&\qquad \times 
\,\d x_1\, \d x_2\,\d u\,\G(\d r_1)\,\G(\d r_2)\\
&=\gamma\,\G(C)^2\int_0^\infty\int_0^\infty f(t_1)\sqrt{{\bar F}_B(t_1)}f(t_2)\sqrt{{\bar F}_B(t_2)} \,\BE |B_{t_1,R}|_d |B_{t_2,R}|_d\,\d t_1\,\d t_2.
\end{align*}
Choose $c_B>0$ such that $B\subset c_BB^d$. Then $| B_{t,R}|_d\le \kappa_d(c_Bt+R)^d$ and hence the Cauchy-Schwarz inequality, the convexity of $s\mapsto s^p$, $p\ge 1$, and $\sqrt{a+b}\le \sqrt{a}+\sqrt{b}$, $a,b\ge 0$, yield that
\begin{align*}
\BE | B_{t_1,R}|_d| B_{t_2,R}|_d &\le c_1 \sqrt{\BE (c_Bt_1+R)^{2d}}\sqrt{\BE (c_Bt_2+R)^{2d}}\\
&\le c_2\left(c_B^dt_1^d+\sqrt{\BE R^{2d}}\right)\left(c_B^dt_2^d+\sqrt{\BE R^{2d}}\right)\\
&\le c_3(1+t_1^d)(1+t_2^d),
\end{align*}
where $c_1,c_2,c_3$ denote finite constants independent of the expectation or $t_1,t_2$. From this and  \eqref{assumption} it follows again that the first summand is finite. 

Since $d_B\bigl(x_2,B(u,r_2)\bigr) \le t_1$ if and only if $u\in x_2+B_{t_1,r_2}$, applying Fubini's theorem and  \eqref{refB} (twice) we obtain for the second summand that 
\begin{align*}
&\int_C\int_C\int_{\R^d}\int_{\R^d}\int_{\R^d}
 \frac{f\bigl(d_B(x_1,B(\o,r_1))\bigr)}{h_B\bigl(d_B(x_1,B(\o,r_1)),r_1\bigr)}
 \frac{f\bigl(d_B(x_2,B(\o,r_2))\bigr)}{h_B\bigl(d_B(x_2,B(\o,r_2)),r_2\bigr)} \\
&\qquad \times \sqrt{{\bar F}_B\bigl(d_B(x_1,B(\o,r_1))\bigr){\bar F}_B\bigl(d_B(x_2,B(\o,r_2))\bigr)} \\
&\qquad \times \I\bigl\{d_B(x_2,B(u,r_2)) \le d_B(x_1,B(\o,r_1))\bigr\}\\
&\qquad \times \,\d x_1\, \d x_2\,\d u\,\G(\d r_1)\,\G(\d r_2)\\
&=\int_C\int_C\int_0^\infty\int_0^\infty f(t_1)\sqrt{\bar{F}_B(t_1)}f(t_2)\sqrt{\bar{F}_B(t_2)}
|B_{t_1,r_2}|_d\, \d t_1\, \d t_2\, \G(\d r_1)\,\G(\d r_2),
\end{align*}
which is finite by the same reasoning as above. 

The third summand can be treated in exactly the same way.

To prove positivity of the asymptotic variance we use the
fact that the variance of any square-integrable function
$H(\Psi)$ of the Poisson process
$\Psi:=\{(\xi_n,R_n):n\ge 1\}$ satisfies the inequality
\begin{align*}
\var H(\Psi)\ge \gamma 
 \int^\infty_0\int_{\R^d}
\big(\BE \left[H(\Psi\cup\{(y,r)\})-H(\Psi)\right]\big)^2\,\d y\,\G(\d r);
\end{align*}
see, e.g.,\ \cite[Theorem 4.2]{LaPe11}. In our case this means
that
\begin{align}\label{vun}
\var \eta_W(C)\ge \gamma \int^\infty_0\int_{\R^d}\tilde h(y,r)^2 \,\d y\,\G(\d r),
\end{align}
where
\begin{align*}
\tilde h(y,r):=\,&
\BE\int_W \bigl[g\bigl(d_B(x,Z\cup B(y,r)),r_B(x,Z\cup B(y,r))\bigr)-
g\bigl(d_B(x,Z),r_B(x,Z)\bigr)\bigr]\,\d x\\
=\,&\BE\int_W \I\bigl\{d_B(x,B(y,r))<d_B(x,Z)\bigr\}
\bigl[g\bigl(d_B(x,B(y,r)),r\bigr)-g\bigl(d_B(x,Z),r_B(x,Z)\bigr)\bigr]\,\d x\\
=\,&\BE\int_W \I\bigl\{d_B(\o,B(y-x,r))<d_B(\o,Z)\bigr\}\\
&\qquad\quad \times \bigl[g\bigl(d_B(\o,B(y-x,r)),r\bigr)-g\bigl(d_B(\o,Z),r_B(\o,Z)\bigr)\bigr]\,\d x .
\end{align*}
Here the last identity follows from the stationarity of $Z$ and $g$ is as defined in \eqref{shortg}.
By \eqref{g},
\begin{align*}
\tilde h(y,r)&=\gamma \int_W\int^\infty_0\int^\infty_0 \I\bigl\{d_B(\o,B(y-x,r))<t\bigr\}\\
&\qquad\quad\times \bigl[g\bigl(d_B(\o,B(y-x,r)),r\bigr)-g(t,s)\bigr]
h_B(t,s){\bar F}_B(t)\,\d t\,\G(\d s)\,\d x.
\end{align*}
Assume now that $0<\G(C)<1$ and let $C':=\R^+\setminus C$.
Recalling the definition \eqref{shortg} of $g$, we obtain from \eqref{vun} that
\begin{align*}
\var \eta_W(C)\ge \gamma
\int^\infty_0\int_{\R^d} h^*(y,r)^2
\I\{r\in C',y\in W\} \,\d y\,\G(\d r),
\end{align*}
where 
\begin{align*}
h^*(y,r):=\,&\gamma\int_W\int^\infty_0\int^\infty_0 \I\bigl\{d_B(\o,B(y-x,r))<t\bigr\}
g(t,s)h_B(t,s){\bar F}_B(t)\,\d t\,\G(\d s)\,\d x\\
=\,&\gamma\,\G(C)\int_W\int^\infty_0 \I\bigl\{d_B(\o,B(y-x,r))<t\bigr\}
f(t){\bar F}_B(t)\,\d t\,\d x.
\end{align*}
Applying Jensen's inequality with the normalization of
$\I\{r\in C',y\in W\} \,\d y\,\G(\d r)$, we get
\begin{align*}
\var \eta_W(C)\ge \frac{\gamma}{\G(C')|W|_d}\left(\int_{C'}\int_{W} h^*(y,r)
\,\d y\,\G(\d r)\right)^2.
\end{align*}
Letting $a:=\gamma^3 \G(C)^2/\G(C')>0$ we obtain that
\begin{align*}
&\frac{\var \eta_W(C)}{|W|_d}\\
&\quad\ge 
\frac{a}{|W|^2_d}\left(\int_{C'}\int_{W}\int_W\int^\infty_0 \I\bigl\{d_B(\o,B(y-x,r))<t\bigr\}
f(t){\bar F}_B(t)\,\d t\,\d x \,\d y\,\G(\d r)\right)^2\\
&\quad=
\frac{a}{|W|^2_d}\left(\int_{C'}\int_{\R^d}\int^\infty_0 |W\cap(W-y)|_d\I\bigl\{d_B(\o,B(y,r))<t\bigr\}
f(t){\bar F}_B(t)\,\d t \,\d y\,\G(\d r)\right)^2.
\end{align*}
Hence it is sufficient to show that
\begin{align*}
\int_{C'}\int^\infty_0 \left(\int_{\R^d}\I\bigl\{d_B(\o,B(y,r))<t\bigr\}
\,\d y\right)\, f(t){\bar F}_B(t) \,\d t  \,\G(\d r)>0.
\end{align*}
This is true, since the inner integral is positive
for all $r,t>0$ and since both  
$\int^\infty_0f(t){\bar F}_B(t)\,\d t$ and $\G(C')$ are positive.
\end{proof}

\begin{remark}\rm
The assumption \eqref{assumption} is slightly stronger than \eqref{beta}. 
\end{remark}

\begin{remark}\rm
Let $\widehat{\G}_n(C)$ be given by \eqref{estimator2} with $W=W_n$. 
Theorem \ref{asvar} implies that $\widehat{\G}_n(C)$ is asymptotically weakly 
consistent. Indeed, \eqref{Eeta} and
\[
\frac{\var \eta_{W_n}(C)}{|W_n|_d^2} \longn 0
\]
ensure that ${\eta_{W_n}(C)}/{|W_n|_d}$ converges to $\gamma\,\beta\,\G(C)$
in probability as $n\to\infty$. Especially,
\begin{equation}
\frac{\eta_{W_n}(\R^+)}{|W_n|_d} \longn \gamma\,\beta \quad \hbox{in probability}.
\label{etaW2}
\end{equation}
Hence, by the continuous mapping theorem, ${\eta_{W_n}(C)}/{\eta_W(\R^+)}$ 
converges to $\G(C)$ in probability as $n \to \infty$. This is in accordance
with the following proposition which even shows that $\widehat{\G}_n(C)$ is asymptotically 
strongly consistent.
\end{remark}

\begin{proposition}\label{ascons} 
For any Borel set $C\subset \R^+$, we 
have $\widehat{\G}_n(C) \longn \G(C)$ $\BP$-a.s.
\end{proposition}

\begin{proof} 
The mapping $W\mapsto \eta_W(C)$ defined by \eqref{etaW} is a random measure
on $\R^d$ depending on the Boolean model $Z$
in a translation-invariant way. As the Boolean model is
ergodic (see \cite[Theorem 9.3.5]{SW08}) we can
apply the spatial ergodic theorem (see \cite[Corollary 10.19]{Kallenberg})
to conclude that
$$
\lim_{n\to\infty}|W_n|_d^{-1}\eta_{W_n}(C) = \BE\eta_{[0,1]^d}(C)=\gamma\,\beta\,\G(C)\quad \BP\text{-a.s.}
$$
Applying this  to the numerator as well as to the denominator in 
\eqref{estimator2}, we obtain the desired result.
\end{proof}

\section{Asymptotic normality} \label{sec:normality}
\setcounter{equation}{0}

In this section we study the asymptotic normality of the ratio-unbiased
estimator \eqref{estimator2} for the radius distribution $\G$ of our stationary
Boolean model $Z$ with spherical grains. 
The proof will be based on approximation by $m$-dependent random fields. 
This idea comes from \cite{HM99}, where the same technique was used to prove 
the central limit theorem for random measures which are associated with the Boolean model in an additive way. 
In contrast to \cite{HM99}, the contribution of an individual grain to the random 
measure $A\mapsto\eta_A(C)$ is not determined by the grain alone, but 
does depend on a random number of other grains in a non-trivial manner. 
Therefore the results of \cite{HM99} do not apply in our setting.

We consider, for $n\in\N$ and a Borel set $C\subset \R^+$, the estimator
\[
\widehat{\G}_n(C) = \frac{\eta_{W_n}(C)}{\eta_{W_n}(\R^d)},
\]
where $W_n := [-n,n)^d$ and $\eta_{W_n}$ is given in \eqref{etaW}.
First we concentrate on the asymptotic normality of the numerator $\eta_{W_n}(C)$. In addition to \eqref{beta}, 
we shall need the 
integrability condition
\begin{equation}
\int_0^\infty (1+t^d) f(t)\,\d t < \infty,
\label{assumptionf}
\end{equation}
which is more restrictive than \eqref{assumption}. 

\begin{theorem} \label{clt1}
Assume that \eqref{beta} and  \eqref{assumptionf} are fulfilled. 
Then, for any Borel set $C\subset \R^+$,  
\[
\sqrt{|W_n|_d}\left(\frac{\eta_{W_n}(C)}{|W_n|_d} - \gamma\,\beta\,\G(C)
\right) \longnd N\bigl(0,\sigma^2(C)\bigr),
\]
where $\sigma^2(C)$ is given by \eqref{asvarC}.
\end{theorem}

\begin{proof}
We fix a Borel set $C \subset \R^+$ and skip the dependence on $C$ in
the notation.
Let $E_z := [0,1)^d + z$ for $z \in I_n := \{-n,\ldots,n-1\}^d$.
Then
\[
\eta_{W_n} = \sum_{z \in I_n} \eta_{E_z} = \sum_{z \in I_n} \int_{E_z} 
g\bigl(d_B(x,Z),r_B(x,Z)\bigr)\,\d x,
\]
where $g$ is given by \eqref{shortg}.
For some fixed integer $m$, we put $F_z := E_z \oplus [-m,m)^d$.
We decompose $\eta_{E_z}$ into two random variables 
\[
\eta_z^{(m)} := \int_{E_z} g\bigl(d_B(x,Z(F_z)),r_B(x,Z(F_z))\bigr)\,\d x
\]
and $\tilde\eta_z^{(m)} := \eta_{E_z} - \eta_z^{(m)}$. Let $\eta_{W_n}^{(m)} := 
\sum_{z \in I_n} \eta_z^{(m)}$ and $\tilde\eta_{W_n}^{(m)} := \sum_{z \in I_n}
\tilde\eta_z^{(m)}$ so that $\eta_{W_n}=\eta_{W_n}^{(m)}+\tilde\eta_{W_n}^{(m)}$. 
It is easily seen that 
$\{\eta_u^{(m)}: u \in U\}$ and $\{\eta_v^{(m)}: v \in V\}$ are independent
whenever $U,V \subset \Z^d$ are such that 
$\|u-v\|_\infty > 2m$ for each $u \in U$ and $v \in V$.
Thus, the random variables $\eta_z^{(m)}$, for 
$z \in \Z^d$, constitute a stationary $(2m)$-dependent random field (cf.~\cite[Section 4.3.1]{Heinrich2013}). 
The variance of $\eta_{W_n}^{(m)}$ is
\[
\var \eta_{W_n}^{(m)} = \var \sum_{z \in I_n} \eta_z^{(m)} = \sum_{z_1 \in I_n}
\sum_{z_2 \in I_n} \cov(\eta_{z_1}^{(m)},\eta_{z_2}^{(m)})
= \sum_{z \in I_n-I_n} N_n(z)\cov(\eta_\o^{(m)},\eta_z^{(m)}),
\]
where $N_n(z)$ is the cardinality of $\{(z_1,z_2) \in I_n \times I_n: z_2-z_1=z\}$, which may be
bounded by $|W_n|_d = (2n)^d$ and $\lim_{n \to \infty} N_n(z)/|W_n|_d = 1$ for any $z \in \Z^d$.
We define
\[
(\sigma_n^{(m)})^2 := \frac{\var \eta_{W_n}^{(m)}}{|W_n|_d}.
\]
Since $\cov(\eta_\o^{(m)},\eta_z^{(m)}) = 0$ for $\|z\|>2m$, the limit of $(\sigma_n^{(m)})^2$ 
as $n\to\infty$ 
exists and satisfies
\begin{equation}
(\sigma^{(m)})^2 := \lim_{n \to \infty} (\sigma_n^{(m)})^2 = 
\sum_{z \in \{-2m,\ldots,2m\}^d} \cov\bigl(\eta_\o^{(m)},\eta_z^{(m)}\bigr).
\label{sigman}
\end{equation}
Next we show that $\BE (\eta_\o^{(m)})^2< \infty$. We put $A:=[-m,m+1)^d$, hence
$$
\eta_\o^{(m)}=\int_{E_\o}g\bigl(d_B(x,Z(A)),r_B(x,Z(A))\bigr)\,\d x.
$$
Proceeding as in the proof of Proposition \ref{p1} and bounding $\bar{F}_B^{A,A}(\cdot)$ as well as $(1-\mathbf{1}\{\cdot\}
\mathbf{1}\{\cdot\})$ by $1$, we get
\begin{align*}
\BE (\eta_\o^{(m)})^2&\le \gamma \int_{E_\o}\int_{E_\o}\int_0^\infty\int_A
g\bigl(d_B(y,B(x_1,r)),r\bigr)g\bigl(d_B(y,B(x_2,r)),r\bigr)\, \d y\, \mathbb{G}(\d r)\, \d x_1\, \d x_2\\
&\quad +\gamma^2
\int_{E_\o}\int_{E_\o}\int_0^\infty\int_0^\infty\int_A\int_A
g\bigl(d_B(y_1,B(x_1,r_1)),r_1\bigr)\\
&\qquad\quad\, \times g\bigl(d_B(y_2,B(x_2,r_2)),r_2\bigr)\,\d y_1\, \d y_2\, \mathbb{G}(\d r_1)\,\mathbb{G}(\d r_2)\,
\, \d x_1\, \d x_2.
\end{align*}
The right-hand side increases if $A$ is replaced by $\R^d$. Arguing then as in the proof of Theorem \ref{asvar}, we obtain
$$
\BE (\eta_\o^{(m)})^2\le \gamma|E_\o|_d\left(\int_0^\infty f(t)\, \d t\right)^2\mathbb{G}(C)+
 \left(\gamma|E_\o|_d\int_0^\infty f(t)\, \d t\, \mathbb{G}(C)\right)^2<\infty.
$$
Therefore, the central limit theorem for
stationary $m$-dependent random fields (see, e.g., \cite{Rosen}) yields that
\[
\frac{1}{\sqrt{|W_n|_d}} \sum_{z \in I_n} \left(\eta_z^{(m)}-\BE\eta_z^{(m)}\right)  \longnd N(0,(\sigma^{(m)})^2).
\]
In view of \cite[Theorem 3.2]{Billingsley}, 
it remains to verify that
\begin{equation}
\lim_{m \to \infty} \sigma^{(m)} = \sigma(C)
\label{sigmam}
\end{equation}
and
\begin{equation}
\lim_{m \to \infty} \limsup_{n \to \infty} \BP\left(\frac{1}{\sqrt{|W_n|_d}}\Bigl|\sum_{z \in I_n} (\tilde\eta_z^{(m)}-\BE\tilde\eta_z^{(m)})\Bigr| \geq \varepsilon\right)=0
\quad \hbox{for any $\varepsilon>0$}.
\label{remainder}
\end{equation}
Define
\[\sigma_n^2 := \frac{\var\eta_{W_n}}{|W_n|_d}.\]
Then
\[
|\sigma(C) - \sigma^{(m)}| \leq |\sigma(C) - \sigma_n| + |\sigma_n - \sigma_n^{(m)}| + |\sigma_n^{(m)} - \sigma^{(m)}|.
\]
The first term goes to zero as $n \to \infty$ by Theorem \ref{asvar}, the last term goes to zero as $n \to \infty$ as well, 
for any $m \in \N$, by \eqref{sigman}.
By  Minkowski's inequality, the middle term can be bounded as
\[
|\sigma_n - \sigma_n^{(m)}| \leq \frac{1}{\sqrt{|W_n|_d}} \sqrt{\var \tilde\eta_{W_n}^{(m)}}. 
\]
Therefore, \eqref{sigmam} follows if we can show that
\begin{equation}
\sup_{n \in \N} \frac{1}{|W_n|_d} \var \tilde\eta_{W_n}^{(m)} \longm 0.
\label{vartilde}
\end{equation}
By Chebyshev's inequality, \eqref{vartilde} also implies \eqref{remainder}.
The variance in \eqref{vartilde} satisfies
\[
\frac{1}{|W_n|_d} \var \sum_{z \in I_n} \tilde\eta_z^{(m)} 
= \frac{1}{|W_n|_d} \sum_{z \in I_n-I_n} N_n(z)\cov(\tilde\eta_\o^{(m)},\tilde\eta_z^{(m)})
\leq \sum_{z \in \Z^d} |\cov \bigl(\tilde\eta_\o^{(m)},\tilde\eta_z^{(m)}\bigr)|.
\]
Therefore, the proof will be finished when we show that
\[
\sum_{z \in \Z^d} |\cov \bigl(\tilde\eta_\o^{(m)},\tilde\eta_z^{(m)}\bigr)| \longm 0.
\]

Consider a fixed $z \in \Z^d$. Then the covariance can be written as
\begin{align*}
&\cov \bigl(\tilde\eta_\o^{(m)},\tilde\eta_z^{(m)}\bigr) =
\cov\bigl(\eta_{E_\o},\eta_{E_z}\bigr)-\cov\bigl(\eta_\o^{(m)},\eta_{E_z}\bigr)
-\cov\bigl(\eta_{E_\o},\eta_z^{(m)}\bigr)+\cov\bigl(\eta_\o^{(m)},\eta_z^{(m)}\bigr) \\
&\quad = \int_{E_\o}\int_{E_z} \left[c_{\R^d,\R^d}(x_1,x_2)-c_{F_\o,\R^d}(x_1,x_2)-c_{\R^d,F_z}(x_1,x_2)+c_{F_\o,F_z}(x_1,x_2)\right]
\,\d x_2\,\d x_1,
\end{align*}
where, for Borel sets $A_1,A_2 \subset \R^d$ and $x_1,x_2\in\R^d$,
\[
c_{A_1,A_2}(x_1,x_2) := \cov\Bigl(g\bigl(d_B(x_1,Z(A_1)),r_B(x_1,Z(A_1))\bigr),g\bigl(d_B(x_2,Z(A_2)),r_B(x_2,Z(A_2))\bigr)\Bigr)
\]
is expressed in Lemma \ref{l1} as
\begin{align*}
c_{A_1,A_2}(x_1,x_2) &= \gamma\int_0^\infty\int_{\R^d} 
g\bigl(d_B(x_1,B(y,r)),r\bigr)g\bigl(d_B(x_2,B(y,r)),r\bigr)I_1(A_1,A_2)\,\d y\,\G(\d r)\\
&\quad + \gamma^2 \int_0^\infty \int_0^\infty \int_{\R^d} \int_{\R^d} 
g\bigl(d_B(x_1,B(y_1,r_1)),r_1\bigr)g\bigl(d_B(x_2,B(y_2,r_2)),r_2\bigr)\\
&\qquad \times I_2(A_1,A_2)\,\d y_1\,\d y_2\,\G(\d r_1)\,\G(\d r_2).
\end{align*}
Here we skip the arguments $x_1,x_2,y,r$, respectively $x_1,x_2,y_1,y_2,r_1,r_2$, of  the functions $I_1(A_1,A_2)$ and
$I_2(A_1,A_2)$,  which were defined before Lemma \ref{l1}. 
We shall treat both parts of $c_{A_1,A_2}(x_1,x_2)$ separately. Our aim is to prove that 
\begin{align*}
S_1 &:= \sum_{z \in \Z^d} \int_{E_\o} \int_{E_z}
\int_0^\infty\int_{\R^d} 
g\bigl(d_B(x_1,B(y,r)),r\bigr)g\bigl(d_B(x_2,B(y,r)),r\bigr) \\
&\qquad\times\left|I_1(\R^d,\R^d)-I_1(F_\o,\R^d)-I_1(\R^d,F_z)+I_1(F_\o,F_z)\right|\,\d y\,\G(\d r)\,\d x_2
\,\d x_1
\end{align*}
and
\begin{align*}
S_2 &:= \sum_{z \in \Z^d} \int_{E_\o} \int_{E_z}
\int_0^\infty\int_0^\infty\int_{\R^d}\int_{\R^d} 
g\bigl(d_B(x_1,B(y_1,r_1)),r_1\bigr)g\bigl(d_B(x_2,B(y_2,r_2)),r_2\bigr) \\
&\qquad\times\left|I_2(\R^d,\R^d)-I_2(F_\o,\R^d)-I_2(\R^d,F_z)+I_2(F_\o,F_z)\right|\,\d y_1\,\d y_2\,\G(\d r_1)\,
\,\G(\d r_2)\,\d x_2\,\d x_1
\end{align*}
tend to zero as $m \to \infty$. Observe that $S_1,S_2$ depend on $m$ via the dependence of $F_\o$, $F_z$  on $m$. 

First, we consider $S_1$. We rewrite
\begin{align*}
&I_1(\R^d,\R^d)-I_1(F_\o,\R^d)-I_1(\R^d,F_z)+I_1(F_\o,F_z) \\
&\quad = \I\{y \in F_\o^c \cap F_z^c\}{\bar F}_B^{\R^d,\R^d}(x_1,x_2;t_1,t_2) 
 + \I\{y \in F_\o \cap F_z^c\}\left({\bar F}_B^{\R^d,\R^d}-{\bar F}_B^{F_\o,\R^d}\right)
(x_1,x_2;t_1,t_2) \\
&\quad\quad
+ \I\{y \in F_\o^c \cap F_z\}\left({\bar F}_B^{\R^d,\R^d}-{\bar F}_B^{\R^d,F_z}\right)(x_1,x_2;t_1,t_2) \\
&\quad\quad + \I\{y \in F_\o \cap F_z\}\left({\bar F}_B^{\R^d,\R^d}-{\bar F}_B^{F_\o,\R^d}-{\bar F}_B^{\R^d,F_z}+{\bar F}_B^{F_\o,F_z}
\right)(x_1,x_2;t_1,t_2)
\end{align*}
with $t_1 = d_B\bigl(x_1,B(y,r)\bigr)$ and $t_2 = d_B\bigl(x_2,B(y,r)\bigr)$.
For notational simplicity, write
\[
\nu_1(x_1,t_1) := \gamma\BE \bigl|(x_1+B_{t_1,R}) \cap F_\o^c\bigr|_d, \quad
\nu_2(x_2,t_2) := \gamma\BE \bigl|(x_2+B_{t_2,R}) \cap F_z^c\bigr|_d
\]
for $x_1,x_2 \in \R^d$ and $t_1,t_2 \in \R^+$. We suppress the dependence on $z$
in $\nu_2(x_2,t_2)$.
From \eqref{empty2exp} and the inequality $1-\e^{-a} \leq a$, for $a \ge 0$, 
we obtain for  $x_1,x_2 \in \R^d$ and $t_1,t_2 \in \R^+$ that
\begin{align}
&\left|\left({\bar F}_B^{\R^d,\R^d}-{\bar F}_B^{F_\o,\R^d}\right)(x_1,x_2;t_1,t_2)\right| =
\left({\bar F}_B^{F_\o,\R^d}-{\bar F}_B^{\R^d,\R^d}\right)(x_1,x_2;t_1,t_2) 
\nonumber \\
&\quad = {\bar F}_B^{F_\o,\R^d}(x_1,x_2;t_1,t_2)\Bigl(1-\exp\left\{-\gamma \BE \bigl|(x_1+B_{t_1,R}) \cap
(x_2+B_{t_2,R})^c \cap F_\o^c\bigr|_d\right\}\Bigr) \nonumber \\
&\quad\leq \nu_1(x_1,t_1). \label{ineq1a}
\end{align}
Analogously,
\begin{equation}
\left|\left({\bar F}_B^{\R^d,\R^d}-{\bar F}_B^{\R^d,F_z}\right)(x_1,x_2;t_1,t_2)\right|
\leq \nu_2(x_2,t_2).
\label{ineq1b}
\end{equation}
Furthermore,
\begin{align}
&\left({\bar F}_B^{\R^d,\R^d}-{\bar F}_B^{F_\o,\R^d}-{\bar F}_B^{\R^d,F_z}+
{\bar  F}_B^{F_\o,F_z}\right)(x_1,x_2;t_1,t_2)= {\bar F}_B^{F_\o,F_z}(x_1,x_2;t_1,t_2)\nonumber\\
&\quad \times \Bigl(1-\exp\left\{-\gamma \BE
\bigl|(x_1+B_{t_1,R}) \cap F_\o^c \cap \bigl((x_2+B_{t_2,R})^c \cup F_z^c\bigr)\bigr|_d\right\}\Bigr) \nonumber \\
&\quad \times \Bigl(1-\exp\left\{-\gamma \BE \bigl|(x_2+B_{t_2,R}) \cap F_z^c \cap 
\bigl((x_1+B_{t_1,R})^c \cup F_\o^c\bigr) 
\bigr|_d\right\}\Bigr) \nonumber  \\
&\quad + {\bar F}_B^{\R^d,\R^d}(x_1,x_2;t_1,t_2)
\Bigl(1-\exp\left\{-\gamma \BE \bigl|(x_1+B_{t_1,R}) \cap (x_2+B_{t_2,R}) \cap F_\o^c \cap
F_z^c\bigr|_d\right\}\Bigr) \label{F4}
\end{align}
gives
\begin{align}
&\left|\left({\bar F}_B^{\R^d,\R^d}-{\bar F}_B^{F_\o,\R^d}-{\bar F}_B^{\R^d,F_z}+{\bar F}_B^{F_\o,F_z}\right)(x_1,x_2;t_1,t_2)\right| \nonumber \\
&\qquad\qquad\qquad\leq \nu_1(x_1,t_1)\nu_2(x_2,t_2) + \sqrt{\nu_1(x_1,t_1)\nu_2(x_1,t_2)} \label{ineq1c}
\end{align}
because by the Cauchy-Schwarz inequality
\[
\BE |X_1 \cap X_2|_d \leq \sqrt{\BE |X_1|_d} \sqrt{\BE |X_2|_d}
\]
for any random sets $X_1$ and $X_2$.
Combining \eqref{ineq1a}, \eqref{ineq1b} and \eqref{ineq1c}, we obtain
\begin{align}
&\left|I_1(\R^d,\R^d)-I_1(F_\o,\R^d)-I_1(\R^d,F_z)+I_1(F_\o,F_z)\right| \nonumber\\
&\leq \I\{y \in F_\o^c \cap F_z^c\} 
+ \I\{y \in F_\o \cap F_z^c\}\nu_1(x_1,t_1)  + \I\{y \in F_\o^c \cap F_z\}\nu_2(x_2,t_2)
\nonumber \\
&\quad + \I\{y \in F_\o \cap F_z\}\left[\nu_1(x_1,t_1)\nu_2(x_2,t_2) + 
\sqrt{\nu_1(x_1,t_1)\nu_2(x_2,t_2)}\right],
\label{I1four}
\end{align}
where $t_1 = d_B\bigl(x_1,B(y,r)\bigr)$ and $t_2 = d_B\bigl(x_2,B(y,r)\bigr)$.
If $x_1 \in E_\o$ and $x_2 \in E_z$, then $\nu_1(x_1,t_1)$ is bounded by
$\nu(t_1)$ and $\nu_2(x_2,t_2)$ is bounded by $\nu(t_2)$, where
\[
\nu(t) := \gamma \BE \bigl|\bigl(E_\o \oplus B_{t,R}\bigr) \cap F_\o^c\bigr|_d, \quad t \in \R^+.
\]
Let $c_B>0$ be such that $B\subset c_BB^d$. Then
\begin{equation}
\nu(t) \leq \BE\bigl|B_{t,R+\sqrt{d}}\bigr|_d \leq \kappa_d \BE (c_Bt+R+\sqrt{d})^d
\leq c_1(1+t^d),
\label{nubound}
\end{equation}
where $c_1$ is a finite constant that does not depend on $t$.
If $x_1 \in E_\o$ and $y \in F_\o^c$, then $\|x_1-y\| \geq m$ and
$d_B\bigl(x_1,B(y,r)\bigr) \geq (m-r)^+/c_B$.
Similarly, if $x_2 \in E_z$
and $y \in F_z^c$, then $d_B\bigl(x_2,B(y,r)\bigr) \geq (m-r)^+/c_B$.
Let 
\[
\psi(t,r) := \I\left\{t \ge \frac{(m-r)^+}{c_B}\right\}, \quad t,r \in \R^+.
\]
Then, by \eqref{I1four} and the substitutions $x_1-y \rightarrow x_1$ and
$x_2-y \rightarrow x_2$, we get
\begin{align*}
S_1 &\leq \sum_{z \in \Z^d} \int_{\R^d}\int_{\R^d} \int_0^\infty\int_{\R^d}
g\bigl(d_B(x_1,rB^d),r\bigr)g\bigl(d_B(x_2,rB^d),r\bigr) \I\{x_1+y \in E_\o,
x_2+y \in E_z\}  \\
&\qquad \times \Bigl[\psi(d_B(x_1,rB^d),r)\psi(d_B(x_2,rB^d),r)
+ \psi(d_B(x_2,rB^d),r)\nu\bigl(d_B(x_1,rB^d)\bigr) \\
&\qquad\quad  + \psi(d_B(x_1,rB^d),r)\nu\bigl(d_B(x_2,rB^d)\bigr) 
 + \nu\bigl(d_B(x_1,rB^d)\bigr)\nu\bigl(d_B(x_2,rB^d)\bigr) \\
&\qquad\quad + \sqrt{\nu\bigl(d_B(x_1,rB^d)\bigr)\nu\bigl(d_B(x_2,rB^d)\bigr)}\Bigr]
\,\d y\,\G(\d r)\,\d x_2\,\d x_1 ,
\end{align*}
and thus
\begin{align*}
S_1 &\leq \int_0^\infty \int_{\R^d}\int_{\R^d} g\bigl(d_B(x_1,rB^d),r\bigr)g\bigl(d_B(x_2,rB^d),r\bigr)\\
&\qquad\times
v\bigl(d_B(x_1,rB^d),d_B(x_2,rB^d),r\bigr)\,\d x_2\,\d x_1\,\G(\d r),
\end{align*}
where
\[
v(t_1,t_2,r) := \bigl(\psi(t_1,r)+\nu(t_1)\bigr)\bigl(\psi(t_2,r)+\nu(t_2)\bigr) + 
\sqrt{\nu(t_1)\nu(t_2)}, \quad t_1,t_2,r \in \R^+.
\]
Now Fubini's theorem, two applications of \eqref{refB} and definition \eqref{shortg} of $g$ yield
\[
S_1 \leq \int_C \int_0^\infty \int_0^\infty f(t_1)f(t_2)v(t_1,t_2,r)\,\d t_2\,\d t_1
\,\G(\d r).
\]
Our moment assumption \eqref{R2d} ensures that 
$\BE \bigl|E_\o \oplus B_{t,R}\bigr|_d < \infty$ and therefore $\nu(t)\to 0$ 
 as $m \to \infty$ by  Lebesgue's dominated convergence theorem.
By \eqref{nubound} and \eqref{assumptionf}, another application of Lebesgue's dominated convergence theorem
shows that $S_1\to 0$ as $m \to \infty$.

Next, we proceed with $S_2$. From \eqref{vor2.6} and the inequality 
$1-\e^{-a} \leq a$, for $a \geq 0$, we obtain 
for $x_1,x_2 \in \R^d$ and $t_1,t_2 \in \R^+$ that 
\begin{align}
&{\bar F}_B^{\R^d,\R^d}(x_1,x_2;t_1,t_2)-{\bar F}_B^{\R^d}(x_1;t_1){\bar F}_B^{\R^d}(x_2;t_2) \nonumber \\
&\qquad = {\bar F}_B^{\R^d,\R^d}(x_1,x_2;t_1,t_2) \Bigl(1-\exp\{-\gamma\BE
\kappa_B(x_2-x_1;t_1,t_2,R)\}\Bigr) \nonumber \\
&\qquad \leq \gamma \BE \kappa_B(x_2-x_1;t_1,t_2,R), \label{ineq2a}
\end{align}
where $\kappa_B$ is defined in \eqref{kappaB},
and
\begin{align*}
&\left({\bar F}_B^{\R^d,\R^d}-{\bar F}_B^{F_\o,\R^d}\right)(x_1,x_2;t_1,t_2)-
\left({\bar F}_B^{\R^d}-{\bar F}_B^{F_\o}\right)(x_1;t_1){\bar F}_B^{\R^d}(x_2;t_2) \\
&\quad = -{\bar F}_B^{F_\o,\R^d}(x_1,x_2;t_1,t_2)\Bigl(1-\exp\left\{-\gamma \BE
\bigl|(x_1+B_{t_1,R}) \cap (x_2+B_{t_2,R}) \cap F_\o\bigr|_d\right\}\Bigr) \\
&\quad\quad + {\bar F}_B^{\R^d,\R^d}(x_1,x_2;t_1,t_2)
\Bigl(1-\exp\left\{-\gamma \BE \bigl|(x_1+B_{t_1,R}) \cap (x_2+B_{t_2,R})\bigr|_d\right\}\Bigr).
\end{align*}
Therefore,
\begin{align}
&\left|\left({\bar F}_B^{\R^d,\R^d}-{\bar F}_B^{F_\o,\R^d}\right)(x_1,x_2;t_1,t_2)-
\left({\bar F}_B^{\R^d}-{\bar F}_B^{F_\o}\right)(x_1;t_1){\bar F}_B^{\R^d}(x_2;t_2) \right| \nonumber \\
&\qquad \leq \gamma \BE\kappa_B(x_2-x_1;t_1,t_2,R). \label{ineq2b}
\end{align}
Analogously,
\begin{align}
&\left|\left({\bar F}_B^{\R^d,\R^d}-{\bar F}_B^{\R^d,F_z}\right)(x_1,x_2;t_1,t_2)-
{\bar F}_B^{\R^d}(x_1;t_1)\left({\bar F}_B^{\R^d}-{\bar F}_B^{F_z}\right)(x_2;t_2) \right| \nonumber \\
&\qquad \leq \gamma \BE\kappa_B(x_2-x_1;t_1,t_2,R). \label{ineq2c}
\end{align}
Finally, using \eqref{F4} we get
\begin{align*}
&\left({\bar F}_B^{\R^d,\R^d}-{\bar F}_B^{F_\o,\R^d}-{\bar F}_B^{\R^d,F_z}
+{\bar F}_B^{F_\o,F_z}\right)(x_1,x_2;t_1,t_2)\\
&\qquad\qquad\qquad\qquad\qquad\qquad\qquad\qquad -\left({\bar F}_B^{\R^d}-{\bar F}_B^{F_\o}\right)(x_1;t_1)\left({\bar F}_B^{\R^d}-{\bar F}_B^{F_z}\right)(x_2;t_2) \\
&\quad =  {\bar F}_B^{F_\o,F_z}(x_1,x_2;t_1,t_2) \Bigl(1-\exp\left\{-\gamma \BE
\bigl|(x_1+B_{t_1,R}) \cap F_\o^c \cap \bigl((x_2+B_{t_2,R})^c 
\cup F_z^c\bigr)\bigr|_d\right\}\Bigr) \\
&\qquad\quad \times \Bigl(1-\exp\left\{-\gamma \BE \bigl|(x_2+B_{t_2,R}) \cap F_z^c \cap 
\bigl((x_1+B_{t_1,R})^c \cup F_\o^c\bigr)\bigr|_d\right\}\Bigr)  \\
&\qquad + {\bar F}_B^{\R^d,\R^d}(x_1,x_2;t_1,t_2)
\Bigl(1-\exp\left\{-\gamma \BE \bigl|(x_1+B_{t_1,R}) \cap (x_2+B_{t_2,R}) \cap F_\o^c \cap
F_z^c\bigr|_d\right\}\Bigr) \\
&\qquad - {\bar F}_B^{F_\o}(x_1;t_1){\bar F}_B^{F_z}(x_2;t_2)\bigl(1-\exp\{-\nu_1(x_1,t_1)\}\bigr)
\bigl(1-\exp\{-\nu_2(x_2,t_2)\}\bigr) \\
&\quad =  {\bar F}_B^{F_\o,F_z}(x_1,x_2;t_1,t_2) \Bigl(1-\exp\left\{-\gamma \BE
\bigl|(x_1+B_{t_1,R}) \cap F_\o^c \cap \bigl((x_2+B_{t_2,R})^c \cup F_z^c\bigr)\bigr|_d\right\}\Bigr) \\
&\qquad\quad \times \Bigl(1-\exp\left\{-\gamma \BE \bigl|(x_2+B_{t_2,R}) \cap F_z^c
\cap \bigl((x_1+B_{t_1,R})^c \cup F_\o^c\bigr)\bigr|_d\right\}\Bigr)  \\
&\qquad\quad \times \Bigl(1-\exp\left\{-\gamma \BE \bigl|(x_1+B_{t_1,R}) \cap (x_2+B_{t_2,R})
\cap F_\o \cap F_z\bigr|_d\right\}\Bigr) \\
&\qquad + {\bar F}_B^{\R^d,\R^d}(x_1,x_2;t_1,t_2)
\Bigl(1-\exp\{-\gamma \BE \bigl|(x_1+B_{t_1,R}) \cap (x_2+B_{t_2,R}) \cap F_\o^c \cap
F_z^c\bigr|_d\}\Bigr) \\
&\qquad - {\bar F}_B^{F_\o}(x_1;t_1){\bar F}_B^{F_z}(x_2;t_2)
\Bigl[\exp\left\{-\gamma \BE \bigl|(x_2+B_{t_2,R}) \cap F_z^c \cap 
\bigl((x_1+B_{t_1,R})^c \cup F_\o^c\bigr) \bigr|_d\right\} \\
&\qquad\quad \times \bigl(1-\exp\{-\gamma \BE \bigl|(x_1+B_{t_1,R}) \cap (x_2+B_{t_2,R}) 
\cap F_\o \cap F_z^c\bigr|_d\}\bigr)\bigl(1-\exp\{-\nu_1(x_1,t_1)\}\bigr) \\
&\qquad\quad + \exp\left\{-\gamma \BE
\bigl|(x_1+B_{t_1,R}) \cap F_\o^c \cap \bigl((x_2+B_{t_2,R})^c \cup F_z^c\bigr)\bigr|_d\right\} \\
&\qquad\quad \times \bigl(1-\exp\left\{-\gamma \BE \bigl|(x_1+B_{t_1,R}) \cap (x_2+B_{t_2,R}) 
\cap F_\o^c \cap F_z\bigr|_d\right\}\bigr) \\
&\qquad\quad \times \bigl(1-\exp\left\{-\gamma \BE \bigl|(x_2+B_{t_2,R}) \cap F_z^c \cap
\bigl((x_1+B_{t_1,R})^c \cup F_\o^c\bigr) \bigr|_d\right\}\bigr)\Bigr],
\end{align*}
which leads to
\begin{align}
&\left|\left({\bar F}_B^{\R^d,\R^d}-{\bar F}_B^{F_\o,\R^d} -{\bar F}_B^{\R^d,F_z}
+{\bar F}_B^{F_\o,F_z}\right)(x_1,x_2;t_1,t_2)\right.\nonumber\\
& \left.\qquad\qquad\qquad\qquad \qquad\qquad\qquad\qquad
-\left({\bar F}_B^{\R^d}-{\bar F}_B^{F_\o}\right)(x_1;t_1)\left({\bar F}_B^{\R^d}-{\bar F}_B^{F_z}\right)(x_2;t_2)\right|
\nonumber \\
&\quad \leq \nu_1(x_1,t_1)\nu_2(x_2,t_2)\gamma \BE\kappa_B(x_2-x_1;t_1,t_2,R)\nonumber\\
&\quad\quad + 
\gamma\BE\bigl|(x_1+B_{t_1,R}) \cap (x_2+B_{t_2,R})
\cap F_\o^c \cap F_z^c\bigr|_d \nonumber \\
&\quad\quad + \nu_1(x_1,t_1)\gamma \BE\kappa_B(x_2-x_1;t_1,t_2,R)
+ \nu_2(x_2,t_2)\gamma\BE\kappa_B(x_2-x_1;t_1,t_2,R). \label{ineq2d}
\end{align}
In the following, we use  \eqref{ineq2a} -- \eqref{ineq2d} with 
 $t_1 = d_B\bigl(x_1,B(y_1,r_1)\bigr)$,  
$t_2 = d_B\bigl(x_2,B(y_2,r_2)\bigr)$. Moreover, we define
\begin{align*}
\chi_1 &:= \I\left\{d_B\bigl(x_1,B(y_2,r_2)\bigr) \le t_1\right\} = \I\left\{x_1 \in y_2+B_{t_1,r_2}\right\}, \\
\chi_2 &:= \I\left\{d_B\bigl(x_2,B(y_1,r_1)\bigr) \le t_2\right\} = \I\left\{x_2 \in y_1+B_{t_2,r_1}\right\},
\end{align*}
and
\[
{\tilde\kappa}_B(x_1,x_2;t_1,t_2,R) := \bigl|(x_1+B_{t_1,R}) \cap (x_2+B_{t_2,R})
\cap F_\o^c\bigr|_d. 
\]
Next, we fix $x_1 \in E_\o$ and $x_2 \in E_z$, for the moment, and  distinguish several cases.
\begin{enumerate}
\item If $y_1 \in F_\o^c$ and $y_2 \in F_z^c$, then $t_1 \ge (m-r_1)^+/c_B$, $t_2 \ge (m-r_2)^+/c_B$
and using \eqref{ineq2a} we get 
\begin{align*}
&\left|I_2(\R^d,\R^d)-I_2(F_\o,\R^d)-I_2(\R^d,F_z)+I_2(F_\o,F_z)\right| =
\left|I_2(\R^d,\R^d)\right| \\
&\quad = \left|(1-\chi_1)(1-\chi_2){\bar F}_B^{\R^d,\R^d}(x_1,x_2;t_1,t_2)
-{\bar F}_B(t_1){\bar F}_B(t_2)\right| \\
&\quad \leq \gamma\BE\kappa_B(x_2-x_1;t_1,t_2,R) + \chi_1 + \chi_2.
\end{align*}
\item If $y_1 \in F_\o$ and $y_2 \in F_z^c \cap F_\o$, then $t_2 \ge (m-r_2)^+/c_B$ and
using \eqref{ineq2b} we get
\begin{align*}
&\left|I_2(\R^d,\R^d)-I_2(F_\o,\R^d)-I_2(\R^d,F_z)+I_2(F_\o,F_z)\right| =
\left|I_2(\R^d,\R^d)-I_2(F_\o,\R^d)\right| \\
&\quad = \left|(1-\chi_1)(1-\chi_2)\left({\bar F}_B^{\R^d,\R^d}-
{\bar F}_B^{F_\o,\R^d}\right)(x_1,x_2;t_1,t_2)-\left({\bar F}_B^{\R^d}-{\bar F}_B^{F_\o}\right)(x_1;t_1)
{\bar F}_B(t_2)\right| \\
&\quad \le \gamma\BE\kappa_B(x_2-x_1;s_1,s_2,R) + \chi_1 + \chi_2.
\end{align*}
\item If $y_1 \in F_\o$ and $y_2 \in F_z^c \cap F_\o^c$, then $t_2 \ge (m-r_2)^+/c_B$ and
using \eqref{ineq2b} we get
\begin{align*}
&\left|I_2(\R^d,\R^d)-I_2(F_\o,\R^d)-I_2(\R^d,F_z)+I_2(F_\o,F_z)\right| =
\left|I_2(\R^d,\R^d)-I_2(F_\o,\R^d)\right| \\
&\quad = \left|(1-\chi_2)\left((1-\chi_1){\bar F}_B^{\R^d,\R^d}-
{\bar F}_B^{F_\o,\R^d}\right)(x_1,x_2;t_1,t_2)-\left({\bar F}_B^{\R^d}-{\bar F}_B^{F_\o}\right)(x_1;t_1)
{\bar F}_B(t_2)\right| \\
&\quad \le \gamma\BE\kappa_B(x_2-x_1;s_1,s_2,R) + \chi_1 + \chi_2.
\end{align*}
\item If $y_1 \in F_o^c \cap F_z$ and $y_2 \in F_z$, then $t_1 \ge (m-r_1)^+/c_B$ and
using \eqref{ineq2c} we get
\begin{align*}
&\left|I_2(\R^d,\R^d)-I_2(F_\o,\R^d)-I_2(\R^d,F_z)+I_2(F_\o,F_z)\right| =
\left|I_2(\R^d,\R^d)-I_2(\R^d,F_z)\right| \\
&\quad = \left|(1-\chi_1)(1-\chi_2)\left({\bar F}_B^{\R^d,\R^d}-
{\bar F}_B^{\R^d,F_z}\right)(x_1,x_2;t_1,t_2)-{\bar F}_B(t_1)\left({\bar F}_B^{\R^d}-{\bar F}_B^{F_z}\right)(x_2;t_2)
\right| \\
&\quad \le \gamma\BE\kappa_B(x_2-x_1;t_1,t_2,R) + \chi_1 + \chi_2.
\end{align*}
\item If $y_1 \in F_\o^c \cap F_z^c$ and $y_2 \in F_z$, then $t_1 \ge (m-r_1)^+/c_B$ and
using \eqref{ineq2c} we get
\begin{align*}
&\left|I_2(\R^d,\R^d)-I_2(F_\o,\R^d)-I_2(\R^d,F_z)+I_2(F_\o,F_z)\right| =
\left|I_2(\R^d,\R^d)-I_2(\R^d,F_z)\right| \\
&\quad = \left|(1-\chi_1)\left((1-\chi_2){\bar F}_B^{\R^d,\R^d}-
{\bar F}_B^{\R^d,F_z}\right)(x_1,x_2;t_1,t_2)-{\bar F}_B(t_1)\left({\bar F}_B^{\R^d}-{\bar F}_B^{F_z}\right)(x_2;t_2)
\right| \\
&\quad \le \gamma\BE\kappa_B(x_2-x_1;t_1,t_2,R) + \chi_1 + \chi_2.
\end{align*}
\item If $y_1 \in F_\o \cap F_z$ and $y_2 \in F_z \cap F_\o$, then by \eqref{ineq1c} and \eqref{ineq2d},
\begin{align*}
&\left|I_2(\R^d,\R^d)-I_2(F_\o,\R^d)-I_2(\R^d,F_z)+I_2(F_\o,F_z)\right| \\
&\quad = \Bigl|(1-\chi_1)(1-\chi_2)\left({\bar F}_B^{\R^d,\R^d}-{\bar F}_B^{F_\o,\R^d}-{\bar F}_B^{\R^d,F_z}
+{\bar F}_B^{F_\o,F_z}\right)(x_1,x_2;t_1,t_2) \\
&\quad\quad - \left({\bar F}_B^{\R^d}-{\bar F}_B^{F_\o}\right)(x_1;t_1)\left({\bar F}_B^{\R^d}-{\bar F}_B^{F_z}\right)
(x_2;t_2)\Bigr| \\
&\quad \leq \bigl(\nu_1(x_1,t_1)\nu_2(x_2,t_2)+\nu_1(x_1,t_1)+\nu_2(x_2,t_2)\bigr)\gamma \BE\kappa_B(x_2-x_1;t_1,t_2,R) \\
&\quad\quad + \gamma\BE {\tilde\kappa}_B(x_1,x_2;t_1,t_2,R) 
+ (\chi_1 + \chi_2)\left(\nu_1(x_1,t_1)\nu_2(x_2,t_2) + \sqrt{\nu_1(x_1,t_1)\nu_2(x_2,t_2)}\right).
\end{align*}
\item If $y_1 \in F_\o \cap F_z^c$ and $y_2 \in F_z \cap F_\o$, then 
by \eqref{ineq1a}, \eqref{ineq1c} and \eqref{ineq2d},
\begin{align*}
&\left|I_2(\R^d,\R^d)-I_2(F_\o,\R^d)-I_2(\R^d,F_z)+I_2(F_\o,F_z)\right| = \Bigl|(1-\chi_1)(1-\chi_2) \\
&\qquad \times \left({\bar F}_B^{\R^d,\R^d}-{\bar F}_B^{F_\o,\R^d}\right)(x_1,x_2;t_1,t_2) 
-(1-\chi_1)\left({\bar F}_B^{\R^d,F_z}-{\bar F}_B^{F_\o,F_z}\right)(x_1,x_2;t_1,t_2) \\
&\quad\quad - \left({\bar F}_B^{\R^d}-{\bar F}_B^{F_\o}\right)(x_1;t_1)\left({\bar F}_B^{\R^d}-{\bar F}_B^{F_z}\right)
(x_2;t_2)\Bigr| \\
&\quad \leq \bigl(\nu_1(x_1,t_1)\nu_2(x_2,t_2)+\nu_1(x_1,t_1)+\nu_2(x_2,t_2)\bigr)\gamma \BE\kappa_B(x_2-x_1;t_1,t_2,R) \\
&\quad\quad + \gamma\BE {\tilde\kappa}_B(x_1,x_2;t_1,t_2,R) 
 + \chi_1\left(\nu_1(x_1,t_1)\nu_2(x_2,t_2)+\sqrt{\nu_1(x_1,t_1)\nu_2(x_2,t_2)}\right)\\
&\quad\quad +\chi_2\nu_1(x_1,t_1).
\end{align*}
\item If $y_1 \in F_\o \cap F_z$ and $y_2 \in F_z \cap F_\o^c$, then  
by \eqref{ineq1b}, \eqref{ineq1c} and \eqref{ineq2d},
\begin{align*}
&\left|I_2(\R^d,\R^d)-I_2(F_\o,\R^d)-I_2(\R^d,F_z)+I_2(F_\o,F_z)\right| = \Bigl|(1-\chi_1)(1-\chi_2) \\
&\quad\quad \times \left({\bar F}_B^{\R^d,\R^d}-{\bar F}_B^{\R^d,F_z}\right)(x_1,x_2;t_1,t_2)-
(1-\chi_2)\left({\bar F}_B^{F_\o,\R^d}-{\bar F}_B^{F_\o,F_z}\right)(x_1,x_2;t_1,t_2) \\
&\quad\quad - \left({\bar F}_B^{\R^d}-{\bar F}_B^{F_\o}\right)(x_1;t_1)\left({\bar F}_B^{\R^d}-{\bar F}_B^{F_z}\right)
(x_2;t_2)\Bigr| \\
&\quad \leq \bigl(\nu_1(x_1,t_1)\nu_2(x_2,t_2)+\nu_1(x_1,t_1)+\nu_2(x_2,t_2)\bigr)\gamma \BE\kappa_B(x_2-x_1;t_1,t_2,R) \\
&\quad\quad + \gamma\BE {\tilde\kappa}_B(x_1,x_2;t_1,t_2,R) 
+ \chi_1\nu_2(x_2,t_2)\\
&\quad\quad + \chi_2\left(\nu_1(x_1,t_1)\nu_2(x_2,t_2)+\sqrt{\nu_1(x_1,t_1)\nu_2(x_2,t_2)}\right).
\end{align*}
\item If $y_1 \in F_\o \cap F_z^c$ and $y_2 \in F_z \cap F_\o^c$, then $\chi_1 = 1$ implies $t_1 \ge (m-r_2)^+/c_B$
and $\chi_2=1$ implies $t_2 \ge (m-r_1)^+/c_B$, and by \eqref{ineq2d} we have
\begin{align*}
&|I_2(\R^d,\R^d)-I_2(F_\o,\R^d)-I_2(\R^d,F_z)+I_2(F_\o,F_z)| = \Bigl|\Bigl((1-\chi_1)(1-\chi_2){\bar F}_B^{\R^d,\R^d} \\
&\quad\quad -(1-\chi_2){\bar F}_B^{F_\o,\R^d}
-(1-\chi_1){\bar F}_B^{\R^d,F_z}+{\bar F}_B^{F_\o,F_z}\Bigr)(x_1,x_2;t_1,t_2) \\
&\quad\quad - \left({\bar F}_B^{\R^d}-{\bar F}_B^{F_\o}\right)(x_1;t_1)\left({\bar F}_B^{\R^d}-{\bar F}_B^{F_z}\right)
(x_2;t_2)\Bigr| \\
&\quad \leq \bigl(\nu_1(x_1,t_1)\nu_2(x_2,t_2)+\nu_1(x_1,t_1)+\nu_2(x_2,t_2)\bigr)\gamma \BE\kappa_B(x_2-x_1;t_1,t_2,R) \\
&\quad\quad + \gamma\BE {\tilde\kappa}_B(x_1,x_2;t_1,t_2,R) 
+ \chi_1 + \chi_2.
\end{align*}
\end{enumerate}
Altogether this gives
\[
\left|I_2(\R^d,\R^d)-I_2(F_\o,\R^d)-I_2(\R^d,F_z)+I_2(F_\o,F_z)\right| \leq I_2^*(x_1,x_2,y_1,y_2,r_1,r_2),
\]
where
\begin{align*}
&I_2^*(x_1,x_2,y_1,y_2,r_1,r_2) := \bigl(\psi(t_1,r_1)\psi(t_2,r_2)+2\psi(t_2,r_2)+2\psi(t_1,r_1)\bigr) \\
&\quad\quad \times (\gamma\BE\kappa_B(x_2-x_1;t_1,t_2,R) + \chi_1 + \chi_2) \\
&\quad + 4\bigl(\nu(t_1)\nu(t_2)+\nu(t_1)+\nu(t_2)\bigr)
\gamma \BE\kappa_B(x_2-x_1;t_1,t_2,R) + 4\gamma\BE {\tilde\kappa}_B(x_1,x_2;t_1,t_2,R) \\
&\quad + 2(\chi_1 + \chi_2)\left(\nu(t_1)\nu(t_2) + \sqrt{\nu(t_1)\nu(t_2)}\right)\\
&\quad + \chi_2\nu(t_1) + \chi_1\nu(t_2) + \chi_1\psi(t_1,r_2) + \chi_2\psi(t_2,r_1)
\end{align*}
does not depend on $z$.
We recall the dependence 
of $\chi_1,\chi_2$ on $x_1,x_2,y_1,y_2$ and then carry out the substitutions 
$x_1-y_1 \rightarrow x_1$ and $x_2-y_2 \rightarrow x_2$ to get
\begin{align*}
S_2 &\leq \sum_{z \in \Z^d} \int_{\R^d} \int_{\R^d}
\int_0^\infty\int_0^\infty\int_{\R^d}\int_{\R^d} 
g\bigl(d_B(x_1,r_1B^d),r_1\bigr)g\bigl(d_B(x_2,r_2B^d),r_2\bigr)\I\{x_1+y_1 \in E_\o\} \\
&\quad\times \I\{x_2+y_2 \in E_z\}
I_2^*(x_1+y_1,x_2+y_2,y_1,y_2,r_1,r_2)\,\d y_1\,\d y_2\,\G(\d r_1)\,\G(\d r_2)\,\d x_2\,\d x_1 \\
&= \int_{\R^d}\int_{\R^d}\int_0^\infty\int_0^\infty\int_{\R^d}\int_{\R^d} 
g\bigl(d_B(x_1,r_1B^d),r_1\bigr)g\bigl(d_B(x_2,r_2B^d),r_2\bigr)\I\{x_1+y_1 \in E_\o\} \\
&\quad\times I_2^*(x_1+y_1,x_2+y_2,y_1,y_2,r_1,r_2)
\,\d y_1\,\d y_2\,\G(\d r_1)\,\G(\d r_2)\,\d x_2\,\d x_1.
\end{align*}
To the inner integrals we apply the relations
\begin{align*}
\int_{\R^d} \int_{\R^d} \I\{x_1+y_1 \in E_\o\}\chi_1\,\d y_2\,\d y_1
&= \int_{E_\o-x_1}\int_{\R^d} \I\{x_1+y_1-y_2 \in B_{t_1,r_2}\}\,\d y_2\,\d y_1 = |B_{t_1,r_2}|_d, \\
\int_{\R^d} \int_{\R^d} \I\{x_1+y_1 \in E_\o\}\chi_2\,\d y_2\,\d y_1
&= \int_{E_\o-x_1}\int_{\R^d} \I\{x_2+y_2-y_1 \in B_{t_2,r_1}\}\,\d y_2\,\d y_1 = |B_{t_2,r_1}|_d, 
\end{align*}
\begin{align*}
&\int_{\R^d} \int_{\R^d} \I\{x_1+y_1 \in E_\o\}\BE\kappa_B(x_2+y_2-x_1-y_1;t_1,t_2,R)\,\d y_2\,\d y_1
= \BE |B_{t_1,R}|_d |B_{t_2,R}|_d, \\
&\int_{\R^d} \int_{\R^d} \I\{x_1+y_1 \in E_\o\}\BE{\tilde\kappa}_B(x_1+y_1,x_2+y_2;t_1,t_2,R)\,\d y_2\,\d y_1 \\
&\quad = \int_{\R^d} \I\{x_1+y_1 \in E_\o\} \BE |(x_1+y_1+B_{t_1,R}) \cap F_\o^c|_d |B_{t_2,R}|_d
\leq \BE |(E_\o \oplus B_{t_1,R}) \cap F_\o^c|_d |B_{t_2,R}|_d.
\end{align*}
In the last two equations we used \cite[Theorem 5.2.1]{SW08}.
Consequently, \eqref{refB} and \eqref{shortg} yield
\begin{align*}
S_2 &\leq \int_C\int_C\int_0^\infty\int_0^\infty f(t_1)f(t_2)
\Bigl[\bigl(\psi(t_1,r_1)\psi(t_2,r_2)+2\psi(t_2,r_2)+2\psi(t_1,r_1)\bigr) \\
&\quad \times \bigl(\gamma\BE |B_{t_1,R}|_d |B_{t_2,R}|_d
+|B_{t_1,r_2}|_d+|B_{t_2,r_1}|_d\bigr) \\
&\quad + 4\bigl(\nu(t_1)\nu(t_2)+\nu(t_1)+\nu(t_2)\bigr)\gamma\BE |B_{t_1,R}|_d |B_{t_2,R}|_d
+ 4\gamma\BE |(E_\o \oplus B_{t_1,R})\cap F_{\o}^c|_d   |B_{t_2,R}|_d \\
&\quad + 2(|B_{t_1,r_2}|_d+|B_{t_2,r_1}|_d)\left(\nu(t_1)\nu(t_2) + \sqrt{\nu(t_1)\nu(t_2)}\right)
+ |B_{t_2,r_1}|_d\nu(t_1) + |B_{t_1,r_2}|_d\nu(t_2) \\
&\quad + |B_{t_1,r_2}|_d\psi(t_1,r_2) + |B_{t_2,r_1}|_d\psi(t_2,r_1)
\Bigr]\,\d t_2\,\d t_1\,\G(\d r_2)\,\G(\d r_1).
\end{align*}
Using $\BE|B_{t,R}|_d \leq \BE|(E_\o \oplus B_{t,R})|_d \leq c_1(1+t^d)$, cf.~\eqref{nubound},
the Cauchy-Schwarz inequality and assumption \eqref{assumptionf},
it follows from the Lebesgue dominated convergence theorem that $S_2 \rightarrow 0$ as $m \rightarrow \infty$.

\end{proof}

Now we are dealing with the asymptotic normality of $\widehat{\G}_n(C)$.

\begin{theorem} \label{clt2}
Assume that \eqref{assumptionf} is satisfied.
Let $W_n = [-n,n)^d$.
If $C\subset \R^+$ is a Borel set, then
\[
\sqrt{|W_n|_d}\left(\widehat{\G}_n(C) - \G(C)\right) \longnd N\bigl(0,\sigma_\G^2(C)\bigr),
\]
where
\begin{align}\label{9054}
\sigma_\G^2(C) := \frac{1}{\gamma^2\beta^2}
\left[(1-\G(C))\sigma^2(C)+\G(C)\sigma^2(\R^+\setminus C)-\G(C)(1-\G(C))\sigma^2(\R^+)\right]
\end{align}
and $\sigma^2(\cdot)$ is given by \eqref{asvarC}. If $0<\G(C)<1$, then
$\sigma_\G^2(C)>0$.
\end{theorem}

\begin{proof}
Using \eqref{etaW2} and Slutsky's theorem, the weak limit of 
$\sqrt{|W_n|_d}\left(\widehat{\G}_n(C) - \G(C)\right)$
coincides with the weak limit of
\[
Y_n:=\frac{1}{\gamma\,\beta\,\sqrt{|W_n|_d}}
\left(\eta_{W_n}(C) - \eta_{W_n}(\R^+)\G(C)\right).
\]
Observing that
\begin{align*}
\gamma\,\beta\,Y_n 
=\frac{1}{\sqrt{|W_n|_d}}
\int_{W_n} \bigl(\I\{r_B(x,Z)\in C\}-\G(C)\bigr) f\bigl(d_B(x,Z)\bigr)h_B\bigl(d_B(x,Z),r_B(x,Z)\bigr)^{-1}\,\d x,
\end{align*}
we can proceed along the same lines as in the proof of Theorem \ref{clt1}
and obtain
\[
\gamma\,\beta\,Y_n \longnd N\left(0,\gamma^2\beta^2\sigma_\G^2(C)\right),
\]
provided we can identify the asymptotic variance $\sigma_\G^2(C)$ of $Y_n$.
Theorem \ref{asvar} implies that
\begin{align}\label{986}
&\gamma^2\beta^2\lim_{n\to\infty}\var Y_n\nonumber\\
&\qquad =\sigma^2(C)+\G(C)^2\sigma^2(\R^+)
-2\G(C) \lim_{n\to\infty} \frac{1}{|W_n|_d}\cov\bigl(\eta_{W_n}(C),\eta_{W_n}(\R^+)\bigr).
\end{align}
Since $\eta_{W_n}(\cdot)$ is additive, we obtain from Theorem \ref{asvar} that
\begin{align*}
2\lim_{n\to\infty} &\frac{1}{|W_n|_d}\cov\bigl(\eta_{W_n}(C),\eta_{W_n}(\R^+)\bigr)
=2\sigma^2(C)+2\lim_{n\to\infty} \frac{1}{|W_n|_d}
\cov\bigl(\eta_{W_n}(C),\eta_{W_n}(\R^+\setminus C)\bigr)\\
&=2\sigma^2(C)+\lim_{n\to\infty} \frac{1}{|W_n|_d}
\bigl(\var \eta_{W_n}(\R^+)-\var \eta_{W_n}(C)-\var \eta_{W_n}(\R^+\setminus C)\bigr)\\
&=2\sigma^2(C)+\sigma^2(\R^+)-\sigma^2(C)-\sigma^2(\R^+\setminus C)
=\sigma^2(C)+\sigma^2(\R^+)-\sigma^2(\R^+\setminus C).
\end{align*}
Inserting this result into \eqref{986} we obtain \eqref{9054}
upon some simplification.

To prove the last assertion, we define 
$\tilde g(t,s):=(\I\{s\in C\}-\G(C))f(t)h_B(t,s)^{-1}$
and assume that $0<\G(C)<1$. For a convex body $W\subset\R^d$
we need to consider the variance of
$$
H_W:=\int_W \tilde{g}\bigl(d_B(x,Z),r_B(x,Z)\bigr)\,\d x.
$$
As in the proof of the positivity assertion in Theorem \ref{asvar} 
we obtain that
\begin{align}\label{vun22}
\var H_W\ge \gamma \int^\infty_0\int_{\R^d}\tilde h(y,r)^2 \,\d y\,\G(\d r),
\end{align}
where
\begin{align*}
\tilde h(y,r):=
\,&\BE\int_W \I\bigl\{d_B(\o,B(y-x,r))<d_B(\o,Z)\bigr\}
\tilde{g}\bigl(d_B(\o,B(y-x,r)),r\bigr)\,\d x\\
&-\BE\int_W \I\bigl\{d_B(\o,B(y-x,r))<d_B(\o,Z)\bigr\}
\tilde{g}\bigl(d_B(\o,Z),r_B(\o,Z)\bigr)\,\d x.
\end{align*}
By \eqref{g} and the definition of $\tilde{g}$ the second expectation on the
above right-hand side vanishes for all $y\in W$ and $r\ge 0$.
Therefore
\begin{align*}
\tilde h(y,r)=
\int_W \bar{F}_B\bigl(d_B(\o,B(y-x,r))\bigr)
\tilde{g}\bigl(d_B(\o,B(y-x,r)),r\bigr)\,\d x.
\end{align*}
Again as in the proof of Theorem \ref{asvar} we let
$C':=\R^+\setminus C$ and obtain from Jensen's inequality 
and \eqref{vun22} that
\begin{align*}
\frac{\sqrt{\var H_W}}{\sqrt{|W|_d}}&\ge
\frac{c}{|W|_d}\int_{C'} \int_W\int_W
\bar{F}_B\bigl(d_B(\o,B(y-x,r))\bigr)
f\bigl(d_B(\o,B(y-x,r))\bigr)\\
&\quad \times\bigl(h_B\bigl(d_B(\o,B(y-x,r)),r\bigr)\bigr)^{-1}\,\d x \,\d y\,\G(\d r)\\
&=
\frac{c}{|W|_d}\int_{C'} \int_{\R^d}|W\cap (W-y)|_d
\bar{F}_B\bigl(d_B(\o,B(y,r))\bigr)f\bigl(d_B(\o,B(y,r))\bigr)\\
&\quad\times\bigl(h_B\bigl(d_B(\o,B(y,r)),r\bigr)\bigr)^{-1} \,\d y\,\G(\d r),
\end{align*}
where $c>0$ is a constant not depending on $W$.
Hence it is sufficient to show that
\begin{align*}
\int_{C'} \int_{\R^d}
\bar{F}_B\bigl(d_B(\o,B(y,r))\bigr)f\bigl(d_B(\o,B(y,r))\bigr)
\bigl(h_B\bigl(d_B(\o,B(y,r)),r\bigr)\bigr)^{-1} \,\d y\,\G(\d r)>0.
\end{align*}
By \eqref{refB} the above integral equals
$\G(C')\int_{\R^d} \bar{F}_B(t)f(t)\,  \d t$,
which is positive by \eqref{beta}.
\end{proof}

\begin{remark} \label{sigmaG}\rm
After some manipulation we get
\[
\gamma^2\beta^2\sigma_\G^2(C) = \gamma \int_{\R^d} {\tilde\tau}_1(C,u)\,\d u + 
\gamma^2 \int_{\R^d} {\tilde\tau}_2(C,u)\,\d u,
\]
where
\begin{align*}
{\tilde\tau}_1(C,u) &:= \int_0^\infty \int_{\R^d} \frac{f\bigl(d_B(\o,B(x,r))\bigr)}{h_B\bigl(d_B(\o,B(x,r)),r\bigr)}
\frac{f\bigl(d_B(u,B(x,r))\bigr)}{h_B\bigl(d_B(u,B(x,r)),r\bigr)} \\
&\qquad\qquad \times {\bar F}_B^{(2)}\bigl(u;d_B(\o,B(x,r)),d_B(u,B(x,r))\bigr)\bigl(\I\{r \in C\}-\G(C)\bigr)^2
\,\d x\,\G(\d r)
\end{align*}
and
\begin{align*}
&{\tilde\tau}_2(C,u) \\
&:= \int_0^\infty\int_0^\infty \int_{\R^d}\int_{\R^d}
\frac{f\bigl(d_B(x_1,B(\o,r_1))\bigr)}{h_B\bigl(d_B(x_2,B(\o,r_1)),r_1\bigr)}
\frac{f\bigl(d_B(x_2,B(\o,r_2))\bigr)}{h_B\bigl(d_B(x_2,B(\o,r_2)),r_2\bigr)}  \\
&\qquad\times 
\I\left\{d_B\bigl(x_2,B(u,r_2)\bigr)\leq d_B\bigl(x_1,B(\o,r_1)\bigr)\right\}
\I\left\{d_B\bigl(x_1,B(-u,r_1)\bigr) \leq d_B\bigl(x_2,B(\o,r_2)\bigr)\right\} \\
&\qquad\times {\bar F}_B^{(2)}\bigl(u;d_B(x_1,B(\o,r_1)),d_B(x_2,B(\o,r_2))\bigr)
\bigl(\I\{r_1 \in C\}-\G(C)\bigr)\bigl(\I\{r_2 \in C\}-\G(C)\bigr) \\
&\qquad\times \,\d x_1\,\d x_2\,\G(\d r_1)\,\G(\d r_2).
\end{align*}
This relation can be also obtained directly by an analogue
of the proof of Theorem \ref{asvar}.
\end{remark}

\section{The planar case} \label{sec:planar}

We mentioned at the beginning that the estimator $\widehat{\G}$ which we discussed so far
is based on the data $\left\{\bigl(d_B(x,Z),r_B(x,Z)\bigr):x\in W \setminus Z\right\}$ and therefore may require
information from outside the window $W$. To overcome this problem, a common procedure in spatial 
statistics is the so-called {\em minus sampling}, which can be used, e.g., if the radius 
distribution $\G$ is concentrated on an interval $[0,r_0]$, $0<r_0<\infty$. We can avoid 
such a condition by assuming that the function $f$ is concentrated on an interval $[0,\varepsilon]$ with  
$\varepsilon > 0$. If we then assume that $Z$ is observable in a window $W^{(\varepsilon)}$ which contains  
$W \oplus \varepsilon B$, then, for each $x\in W$, we have either $f\bigl(d_B(x,Z)\bigr)=0$ or 
$d_B(x,Z)\le\varepsilon$, in which case the (almost surely unique) contact point $\bigl(x+d_B(x,Z) B\bigr)\cap Z$ 
lies in $W^{(\varepsilon)}$. 

For practical applications, the planar case $d=2$ is particularly important. Also, then, the spherical case $B=B^2$ and the linear
case $B=[0,u]$ (with a given direction $u$) play a major role. For simplicity, in the following considerations we concentrate on the
window $W=[0,1]^2$ and we assume, as explained above, that $f$ is concentrated on $[0,\varepsilon ], \varepsilon > 0$, and that $Z$ is observed in $W^{(\varepsilon )}$. Let $\tilde C_1,\ldots,\tilde C_k$ be the (connected and relatively open) visible arcs in  
$\partial Z\cap W^{(\varepsilon)}$. We need not know whether some of these arcs belong to the same particle. By 
$C_1\subset \tilde C_1,\ldots,C_k\subset \tilde C_k$ 
we denote the corresponding ``effective'' arcs; these consist of the points
$\bigl(x+d_B(x,Z) B\bigr)\cap Z\in \tilde C_i$, $x\in W\setminus Z$, for which $d_B(x,Z)\le\varepsilon$. Let $r_i$ be the radius and 
$l_i$ the length of $C_i$, and let $A_i$ be the set of points $x\in W\setminus Z$ with $d_B(x,Z)\le\varepsilon$ and which project 
onto $C_i$, $i=1,\ldots,k$, in the sense that $(x+d_B(x,Z) B)\cap Z$ consists of a unique point and this point lies in $C_i$, 
for $x\in A_i$. Then our estimator $\widehat{\G}$ is of the form
\[
\widehat{\G} = \frac{1}{\sum_{i=1}^k w_i} \sum_{i=1}^k w_i\delta_{r_i},
\]
where the weight $w_i$ is given by
\[
w_i = \int_{A_i} f\bigl(d_B(x,Z)\bigr)h_B\bigl(d_B(x,Z),r_B(x,Z)\bigr)^{-1}\,\d x.
\]

For $B=B^2$ we have $h_{B^2}(t,r) = 2\pi (t+r)$ (see Remark \ref{Bball}), hence if 
$f(t) = {\varepsilon}^{-1}\I\{t\le\varepsilon\}$ then
\[
w_i = \frac{1}{2\pi\varepsilon}\int_{A_i\cap(C_i+\varepsilon B^2)} \frac{1}{d_{B^2}(x,C_i)+r_i}\ \d x.
\]
If we let $\varepsilon\to 0$, the weights converge to $w_i =  l_i/(2\pi r_i)$ if $r_i>0$ and to $w_i=1$ if $r_i=l_i=0$. Then  the estimator becomes 
\[
\widehat{\G}_o := \left({\sum_{i=1}^k \frac{l_i}{r_i}}\right)^{-1} \sum_{i=1}^k \frac{l_i}{r_i}\delta_{r_i}
\]
with $l_i/r_i$ interpreted as $2\pi$ if $r_i=l_i=0$. 
Notice also that the outer sampling window $W^{(\varepsilon)}$ then shrinks to $W$, so that in the limit only information in $W$ 
is needed. The estimator $\widehat{\G}_o$ was discussed by Hall \cite[Chapter 5.6]{Hall} (more generally, he considered  estimators of $\BE A(R)$, for a given function $A$; $\widehat{\G}_o$ corresponds to the case $A=\I_C$).

For $B=[0,u]$ (with $u\in\{\pm e_1,\pm e_2\}$), assuming (in the linear case) that $\mathbb{G}(\{0\})=0$ and hence $r_i>0$,  and again choosing $f(t) =  {\varepsilon}^{-1}\I\{t\le\varepsilon\}$, 
we get $h_B(t,r) = 2r$ and
\[
w_i = \frac{1}{2\varepsilon r_i}\int_{A_i\cap(C_i+\varepsilon [0,-u])}\,\d x .
\]
This  yields an estimator $\widehat{\G}_{l,u}$ in the limit $\varepsilon\to 0$, which is given by
\[
\widehat{\G}_{l,u} := \left({\sum_{i=1}^k \frac{l_i(u)}{r_i}}\right)^{-1} \sum_{i=1}^k \frac{l_i(u)}{r_i}\delta_{r_i}.
\]
Here, $l_i(u)$ is the length of the projection of the visible part of $C_i$ in direction $u$ (projected onto the line orthogonal 
to $u$). 
The estimator can be improved by combining $u=e_1,-e_1,e_2,-e_2$, 
\[
\widehat{\G}_l := \frac{1}{4}\left(\widehat{\G}_{l,e_1}+\widehat{\G}_{l,-e_1}
+\widehat{\G}_{l,e_2}+\widehat{\G}_{l,-e_2}\right).
\]
For applications, it would be natural to choose $\varepsilon=1$ which yields weights 
\[
w_i =\frac{|A_i|_2}{2r_i}, \quad i=1,\ldots,k,
\]
and gives the estimator
\[
\widehat{\G} = \left(\sum_{i=1}^k\frac{|A_i|_2}{r_i}\right)^{-1} \sum_{i=1}^k \frac{|A_i|_2}{r_i}\delta_{r_i}.
\]
Hence, in this case and with $u=e_1$, information in $[0,2]\times [0,1]$ would be required and the estimation is based on the areas 
of the regions $A_i\subset [0,1]^2$. Of course, the estimation can be again improved by combining the estimators for 
$u=e_1,-e_1,e_2,-e_2$ which are available if $Z$ is observed in $[-1,2]^2$.

If we do not have information from outside $W$, then we may use a minus sampling
approach and replace $W$ by the eroded window $W_{\ominus \varepsilon} := \{x \in W: 
x+\varepsilon B \subset W\}$, i.e.~we consider the following estimator
\[
\widehat{\G}_{\ominus \varepsilon}(C) := \frac{\eta_{W_{\ominus\varepsilon}}(C)}{\eta_{W_{\ominus\varepsilon}}(\R^+)}.
\]
Another possibility would be to use the naive approach which ignores edge effects.  
Then we have the {\em uncorrected estimator}
\[
\widehat{\G}_{{\rm u}}(C) := \frac{\eta_{W,{\rm u}}(C)}{\eta_{W,{\rm u}}(\R^+)},
\]
where
\[
\eta_{W,{\rm u}}(C) := \int_W \I\{r_B(x,Z \cap W)\in C\}\frac{f\bigl(d_B(x,Z \cap W)\bigr)}{h_B\bigl(d_B(x,Z \cap W),r_B(x,Z \cap W)\bigr)}\,\d x.
\]
If $B=[0,u]$, then it can happen that $d_B(x,Z \cap W)=\infty$. In that case we use our convention concerning $f \cdot h_B^{-1}$,
i.e.~the points $x$ satisfying $d_B(x,Z \cap W)=\infty$ do not contribute to $\eta_{W,{\rm u}}(C)$.
Besides minus sampling there exist more sophisticated methods of edge correction in the statistics of spatial point processes.
We adopt the idea of local minus sampling that was originally applied in \cite{Hanisch} to the estimation 
of the nearest neighbour distance distribution function for stationary point processes (see also \cite{HBG1999}).
We use only points that are closer to $Z$ than to the boundary of the window $W$.
This gives the {\em Hanisch type estimator}
\[
\widehat{\G}_{{\rm H}}(C) := \frac{\eta_{W,{\rm H}}(C)}{\eta_{W,{\rm H}}(\R^+)},
\]
where
\[
\eta_{W,{\rm H}}(C) := \int_W \I\{r_B(x,Z)\in C\}\I\{d_B(x,Z) \leq d_B(x,\partial W)\}
\frac{f\bigl(d_B(x,Z)\bigr)}{h_B\bigl(d_B(x,Z),r_B(x,Z)\bigr)}\,\d x.
\]
Note that for $B=[0,u]$ the estimators $\widehat{\G}_{{\rm H}}$ and $\widehat{\G}_{{\rm u}}$ coincide.

In practical applications one has to replace in \eqref{etaW} 
the integration with respect to Lebesgue measure by an integration with
respect to a discrete measure. This still gives a ratio-unbiased
estimator of $\G$.

We compare the performance of the different estimators discussed above through computer simulations. 
We simulate a stationary planar Boolean model with spherical grains, given by \eqref{eq:1.1}.
The observation window $W$ is the unit square $[0,1]^2$.
The distribution $\G$ is assumed to be uniform on $(0.05,0.1)$. 
We approximate the integrals over $W$ by Riemannian sums over
a rectangular grid of points $L_h \cap W$, where 
\[
L_h := \left\{\bigl((k-1/2)h,(l-1/2)h\bigr): k,l \in \N\right\}. 
\]
For our purposes, we choose $h=1/300$.

We take $f(t) = {\varepsilon}^{-1}\I\{t\le\varepsilon\}$ for different
choices of $\varepsilon$ and compare the estimator $\widehat{\G}$, given by \eqref{estimator2}, 
with the estimators $\widehat{\G}_o$ (for spherical $B$) and $\widehat{\G}_l$ 
(for linear $B$) corresponding to the limiting case $\varepsilon \to 0$.
The estimators $\widehat{\G}_{\ominus \varepsilon}$, $\widehat{\G}_{{\rm u}}$
and $\widehat{\G}_{{\rm H}}$ are also evaluated.
For linear $B=[0,u]$ we always combine the corresponding estimators
for $u=e_1,-e_1,e_2,-e_2$, this leads to a noticeable improvement.

The radius distribution $\G$ is uniquely determined by the
distribution function $G(t) = \G([0,t])$, $t \geq 0$.
We measure the quality of the estimators by the Kolmogorov-Smirnov distance 
\[
d_{{\rm KS}}({\widehat G},G) := \sup_{s \geq 0} |{\widehat G}(s)-G(s)| 
\]
and the Cram\'er-von Mises distance
\[
d_{{\rm CvM}}({\widehat G},G) := \int_{0.05}^{0.1} \bigl({\widehat G}(s)-G(s)\bigr)^2\,\frac{\d s}{0.05}.
\]

We have generated 100 independent realizations of the Boolean model $Z$ with
chosen intensity $\gamma$.  
For each realization we have determined several estimators under study. 
The sample means of corresponding Kolmogorov-Smirnov and Cram\'er-von Mises distances 
over 100 simulations are demonstrated in Table \ref{tab1} for $\gamma=25$ and
in Table \ref{tab2} for $\gamma=100$.
The results show that smaller values of $\varepsilon$ are more preferable. The limiting
estimators $\widehat{\G}_o$ and $\widehat{\G}_l$ produced the smallest error. They are
outperformed only in the case of smaller intensity and linear $B$ where our estimator,
given by \eqref{estimator2}, with larger $\varepsilon$, gives better results.
However, this estimator uses also information from outside $W$.
Simulation studies for exponentially distributed radii (not presented) show very similar 
results. A change of resolution $h$ has only a minor influence on the quality of the 
estimators. For intensity $\gamma \gg 100$ the deviation from the radius distribution 
increases which is intuitively clear because many balls are covered so that their radii 
are not available for the estimators.

\begin{table} 
\centering
\begin{tabular}{|c|cc|cc|}
\hline
& \multicolumn{2}{|c|}{$d_{{\rm KS}}$} & \multicolumn{2}{|c|}{$1000 \cdot d_{{\rm CvM}}$} \\
estimator & spherical $B$ & linear $B$ & spherical $B$ & linear $B$ \\ 
\hline
$\widehat{\G}$, $\varepsilon=1$                         & 0.178    & 0.147     & 7.921     & 5.139   \\
$\widehat{\G}$, $\varepsilon=0.05$                      & 0.172    & 0.170     & 7.317     & 7.101   \\
$\widehat{\G}$, $\varepsilon=0.01$                      & 0.172    & 0.172     & 7.295     & 7.292   \\
\hline
$\widehat{\G}_o$ or $\widehat{\G}_l$                      & 0.171    & 0.172     & 7.243   & 7.257   \\
\hline
$\widehat{\G}_{\ominus \varepsilon}$, $\varepsilon=0.05$  & 0.191    & 0.177     & 9.243   & 7.753   \\
$\widehat{\G}_{\ominus \varepsilon}$, $\varepsilon=0.01$  & 0.176    & 0.173     & 7.674   & 7.435   \\
\hline
$\widehat{\G}_{{\rm u}}$, $\varepsilon=1$                 & 0.182    & 0.179     & 8.389   & 7.890   \\
$\widehat{\G}_{{\rm u}}$, $\varepsilon=0.05$              & 0.173    & 0.169     & 7.480   & 7.553   \\
$\widehat{\G}_{{\rm u}}$, $\varepsilon=0.01$              & 0.173    & 0.168     & 7.322   & 7.472   \\
\hline
$\widehat{\G}_{{\rm H}}$, $\varepsilon=1$                 & 0.187    & 0.179     & 9.003   & 7.890   \\
$\widehat{\G}_{{\rm H}}$, $\varepsilon=0.05$              & 0.179    & 0.169     & 8.023   & 7.553   \\
$\widehat{\G}_{{\rm H}}$, $\varepsilon=0.01$              & 0.174    & 0.168     & 7.462   & 7.472   \\
\hline
\end{tabular}
\caption{Sample means of distances between distribution functions computed from 100 realizations of a Boolean model
with intensity $\gamma=25$ and uniform radius distribution on $(0.05,0.1)$.}
\label{tab1}
\end{table}

\begin{table} 
\centering
\begin{tabular}{|c|cc|cc|}
\hline
& \multicolumn{2}{|c|}{$d_{{\rm KS}}$} & \multicolumn{2}{|c|}{$1000 \cdot d_{{\rm CvM}}$} \\
estimator & spherical $B$ & linear $B$ & spherical $B$ & linear $B$ \\ 
\hline
$\widehat{\G}$, $\varepsilon=1$                         &  0.147   &  0.134    &  5.506    & 4.406   \\
$\widehat{\G}$, $\varepsilon=0.05$                      &  0.145   &  0.131    &  5.294    & 4.238   \\
$\widehat{\G}$, $\varepsilon=0.01$                      &  0.132   &  0.128    &  4.276    & 4.008   \\
\hline
$\widehat{\G}_o$ or $\widehat{\G}_l$                      & 0.127    &  0.127    & 3.919   & 3.928   \\
\hline
$\widehat{\G}_{\ominus \varepsilon}$, $\varepsilon=0.05$  & 0.158    &  0.135    &  6.162  & 4.460   \\
$\widehat{\G}_{\ominus \varepsilon}$, $\varepsilon=0.01$  & 0.134    &  0.129    &  4.359  & 4.029   \\
\hline
$\widehat{\G}_{{\rm u}}$, $\varepsilon=1$                 & 0.150    &  0.140    &  5.710  & 4.838   \\
$\widehat{\G}_{{\rm u}}$, $\varepsilon=0.05$              & 0.147    &  0.137    &  5.431  & 4.807   \\
$\widehat{\G}_{{\rm u}}$, $\varepsilon=0.01$              & 0.133    &  0.129    &  4.299  & 4.208   \\
\hline
$\widehat{\G}_{{\rm H}}$, $\varepsilon=1$                 & 0.150    &  0.140    & 5.602   & 4.838   \\
$\widehat{\G}_{{\rm H}}$, $\varepsilon=0.05$              & 0.148    &  0.137    & 5.438   & 4.807   \\
$\widehat{\G}_{{\rm H}}$, $\varepsilon=0.01$              & 0.133    &  0.129    & 4.323   & 4.208   \\
\hline
\end{tabular}
\caption{Sample means of distances between distribution functions computed from 100 realizations of a Boolean model
with intensity $\gamma=100$ and uniform radius distribution on $(0.05,0.1)$.}
\label{tab2}
\end{table}

\bigskip

\noindent
{\em Authors' addresses:}
\vspace{.5cm}\\
\noindent
Daniel Hug, 
Karlsruhe Institute of Technology (KIT), Department of Mathematics,  

e-mail: daniel.hug@kit.edu

\medskip

\noindent
G{\"u}nter Last, Karlsruhe Institute of Technology (KIT), Department of Mathematics, 

e-mail:  guenter.last@kit.edu

\medskip

\noindent
Zbyn\v ek Pawlas, Department of Probability and Mathematical Statistics,
Faculty of Mathematics and Physics, Charles University, Sokolovsk\'a 83,
186$\,$75 Praha 8, Czech Republic, 

e-mail: {pawlas@karlin.mff.cuni.cz}

\medskip

\noindent
Wolfgang Weil, Karlsruhe Institute of Technology (KIT), Department of Mathematics, 

e-mail:  wolfgang.weil@kit.edu


\begin{thebibliography}{99}

\bibitem{Ballani} 
Ballani, F. 
\newblock On second-order characteristics of germ-grain models with convex grains,
\newblock {\em Mathematika} {\bf 53} (2006), 255--285.

\bibitem{Billingsley} 
Billingsley, P. 
\newblock {\em Convergence of Probability Measures}, 2nd edition,
\newblock John Wiley \& Sons, New York, 1999.

\bibitem{Hall} 
Hall, P. 
\newblock {\em Introduction to the Theory of Coverage Processes},
\newblock John Wiley \& Sons, New York, 1988.

\bibitem{Hanisch} 
Hanisch, K.-H. 
\newblock Some remarks on estimators of the distribution function
of nearest neighbour distance in stationary spatial point patterns,
\newblock {\em Statistics} {\bf 15} (1984), 409--412.

\bibitem{HBG1999} 
Hansen,  M.~B., Baddeley, A.~J., Gill, R.~D. 
\newblock  First contact distributions for spatial patterns: regularity and estimation, 
\newblock {\em Adv. Appl. Prob. (SGSA)} {\bf 31} (1999), 15--33.


\bibitem{Heinrich2013}
Heinrich, L. 
\newblock Asymptotic methods in statistics of random point processes, 
\newblock In: Spodarev, E.~(ed) {\em Stochastic Geometry, Spatial Statistics and Random Fields}, 
pp.~115--150, Lecture Notes in Mathematics {\bf 2068}. Springer, Berlin, 2013.


\bibitem{HM99}
Heinrich, L., Molchanov, I.~S. 
\newblock Central limit theorem for a class of random measures
associated with germ-grain models,
\newblock {\em Adv. Appl. Prob. (SGSA)} {\bf 31} (1999), 283--314. 

\bibitem{HW00}
Heinrich, L., Werner, W. 
\newblock Kernel estimation of the diameter distribution in Boolean
models with spherical grains, 
\newblock {\em J. Nonparametr. Statist.} {\bf 12} (2000), 147--176.

\bibitem{HugLast} 
Hug, D., Last, G. 
\newblock {On support measures in Minkowski spaces and contact distributions 
in stochastic geometry,}  
\newblock {\em Ann. Probab.} {\bf 28} (2000), 796--850.

\bibitem{HLW}
Hug, D., Last, G., Weil, W. 
\newblock Generalized contact distributions of inhomogeneous Boolean models,
\newblock {\em Adv. Appl. Prob. (SGSA)} {\bf 34} (2002), 21--47.

\bibitem{Kallenberg} 
Kallenberg, O. 
\newblock {\em Foundations of Modern Probability}, 2nd edition,
\newblock Springer-Verlag, New York, 2002.

\bibitem{KW}
Kiderlen, M., Weil, W.
\newblock Measure-valued valuations and mixed curvature measures of convex bodies,
\newblock {\em Geom.~Dedicata} {\bf 76} (1999), 291--329.

\bibitem{LaPe11}
Last, G., Penrose, M.~D. 
\newblock Fock space representation, chaos expansion and covariance inequalities 
for general Poisson processes, 
\newblock {\em Prob.\ Theory  Rel.\ Fields} {\bf 150} (2011), 663--690. 

\bibitem{Molchanov90} 
Molchanov, I.~S. 
\newblock Estimation of the size distribution of spherical grains in
the Boolean model,
\newblock {\em Biometrical J.} {\bf 32} (1990), 877--886.

\bibitem{Molchanov} 
Molchanov, I.~S.
\newblock{\em Statistics of the Boolean Model for Practitioners and
Mathematicians},
\newblock Wiley, Chichester, 1997.

\bibitem{MolchanovStoyan} 
Molchanov, I.~S., Stoyan, D. 
\newblock Asymptotic properties of estimators for parameters of the Boolean model,
\newblock {\em Adv. Appl. Prob.} {\bf 27} (1994), 63--86.


\bibitem{Rosen}
Ros\'en, B. 
\newblock A note on asymptotic normality of sums of higher-dimensionally
indexed random variables,
\newblock {\em Ark. Mat.} {\bf 8} (1969), 33--43.

\bibitem{Sch93}
Schneider, R. 
\newblock {\em Convex Bodies: The Brunn-Minkowski Theory}, 
\newblock Cambridge University Press, Cambridge, 1993.


\bibitem{SKM}
Stoyan, D., Kendall, W.~S., Mecke, J. 
\newblock {\em Stochastic Geometry and its Applications}, 
\newblock 2nd edition, Wiley, Chichester, 1995.


\bibitem{SW08}
Schneider, R., Weil, W.
\newblock {\em Stochastic and Integral Geometry}, 
\newblock Springer-Verlag, Berlin, 2008.




\end{thebibliography}
\end{document}